\documentclass[reqno]{amsart}

\usepackage[
left=1.25in, right=1.25in]{geometry}
\usepackage[all]{xy}
\usepackage{graphicx}
\usepackage{indentfirst}
\usepackage{bm}
\usepackage{amsthm}
\usepackage{mathrsfs}
\usepackage{latexsym}
\usepackage{amsmath}
\usepackage{amssymb}
\usepackage{color}
\usepackage{hyperref}
\usepackage{microtype}
\usepackage{tikz}

\usepackage{comment}

\usepackage{color}

\usetikzlibrary{calc}
\usetikzlibrary{shapes.geometric}
\usetikzlibrary{arrows}
\usetikzlibrary{decorations.pathmorphing}
\usetikzlibrary{decorations.text}
\usetikzlibrary{decorations.shapes}
\usetikzlibrary{decorations.fractals}
\usetikzlibrary{decorations.footprints}

\newcommand{\ID}{
\raisebox{-.3cm}{
\begin{tikzpicture}
\fill[gray!50] (0,0) rectangle (.6,.8);
\draw (0,0)--++(0,.8);
\draw (.6,0)--++(0,.8);
\end{tikzpicture}} 
}

\newcommand{\JP}{
\raisebox{-.3cm}{
\begin{tikzpicture}
\fill[gray!50] (0,0) arc (180:0:.3);
\fill[gray!50] (0,.8) arc (-180:0:.3);
\draw (0,0) arc (180:0:.3);
\draw (0,.8) arc (-180:0:.3);
\end{tikzpicture}}
}

\newcommand{\be}{\begin{equation}}
\newcommand{\ee}{\end{equation}}
\newcommand{\ben}{\begin{equation*}}
\newcommand{\een}{\end{equation*}}
\begin{document}

\newtheorem{theorem}{Theorem}[section]
\newtheorem{main}{Main Theorem}[section]
\newtheorem{proposition}[theorem]{Proposition}
\newtheorem{corollary}[theorem]{Corollary}
\newtheorem{definition}[theorem]{Definition}
\newtheorem{notation}[theorem]{Notation}
\newtheorem{lemma}[theorem]{Lemma}
\newtheorem{example}[theorem]{Example}
\newtheorem{remark}[theorem]{Remark}
\newtheorem{question}[theorem]{Question}
\newtheorem{conjecture}[theorem]{Conjecture}
\newtheorem{fact}[theorem]{Fact}
\newtheorem*{ac}{Acknowledgements}

\newcommand\M{\mathcal{M}}
\newcommand\N{\mathcal{N}}
\newcommand\PA{\mathscr{P}}
\newcommand\SA{\mathscr{S}}
\newcommand\F{\mathfrak{F}}
\newcommand\D{\mathfrak{D}}
\newcommand\FS{\mathfrak{F}_S}
\newcommand\C{\mathscr{C}}
\newcommand\Z{\mathbb{Z}}

\title{Block Maps and Fourier Analysis}

\author{Chunlan Jiang}
\address{Department of Mathematics\\ Hebei Normal University}
\email{cljiang@mail.hebtu.edu.cn}
\author{Zhengwei Liu}
\address{Department of Mathematics and Department of Physics\\ Harvard University}
\email{zhengweiliu@fas.harvard.edu}
\author{Jinsong Wu}
\address{School of Mathematical Sciences\\ University of Science and Technology of China}
\email{wjsl@ustc.edu.cn}

%


\maketitle
\begin{abstract}
We introduce block maps for subfactors and study their dynamic systems. We prove that the limit points of the dynamic system are positive multiples of biprojections and zero. For the $\Z_2$ case, the asymptotic phenomenon of the block map coincides with that of that 2D Ising model. 
The study of block maps requires a further development of the recent work of the authors on the Fourier analysis of subfactors.
We generalize the notion of sum set estimates in additive combinatorics for subfactors and prove the exact inverse sum set theorem. Using this new method, we characterize the extremal pairs of Young's inequality for subfactors, as well as the extremal operators of the Hausdorff-Young inequality. 
\end{abstract}

\textbf{Keywords:} Subfactors, Fourier analysis, planar algebras,  block maps, quantum groups

\section{Introduction}
Onsager solved the partition function of the 2D Ising model analytically and find an order-disorder phase transition at the critical temperature $T_c=\displaystyle \frac{2}{ln (\sqrt{2}+1)}$ in 1944  \cite{Ons44}. This
critical temperature was observed by Kramers and Wannier using their duality \cite{KraWan}.

Subfactor theory provides a framework to study the quantum symmetry \cite{JonSun,EvaKaw}.
Jones introduced graph planar algebras for finite bipartite graphs \cite{Jon00}.  The 2D Ising model can be generalized from $\mathbb{Z}_2$ to subfactors using graph planar algebras. The partition function can be represented as a planar diagram $D$ in the graph planar algebra $\mathscr{G}$ as shown in Fig.~\ref{Lattice to Graph}. For the 2D Ising model, the bipartite graph is the principal graph of the $\Z_2$ subfactor, i.e., the Dynkin diagram $A_3$. 

The crossing 
\raisebox{-.3cm}{
\begin{tikzpicture}
\fill[gray!50] (0,0)--++(.3,.4)--++(.3,-.4);
\fill[gray!50] (0,.8)--++(.3,-.4)--++(.3,.4);
\draw (0,0)--++(.6,.8);
\draw (.6,0)--++(-.6,.8);
\end{tikzpicture}}
in the planar diagram $D$ is labelled by a 2-box element $B(T,J)$ in $\mathscr{G}_{\bullet}$, determined by the temperature $T$ and the neighborhood interaction $J$ of the lattice model. For the ferromagnetic 2D Ising model, 
$B(T,J)=(e^{\beta}-e^{-\beta})\ID+e^{-\beta} \sqrt{2} \JP$, where $\beta=T^{-1}$ is the inverse temperature.

\begin{figure}\label{Lattice to Graph}
\begin{tikzpicture}
\foreach \x in {0,1,2,3}
{
\draw (\x,0)--+(0,3);
\foreach \y in {0,1,2,3}
{
\draw (0,\y)--+(3,0);
\node at (\x,\y) {$\bullet$};
}}
\fill[white] (-1,-1) rectangle (-.5,4);
\fill[white] (4,-1) rectangle (3.5,4);
\fill[white] (-1,-1) rectangle (4,-.5);
\fill[white] (-1,4) rectangle (4,3.5);

\begin{scope}[shift={(6.3,1.5)},xscale=.8,yscale=.8]

\draw (0,0)--+(-2,0);
\draw (0,0)--+(-.3,.3);
\draw (0,0)--+(-.3,-.3);
\end{scope}
\begin{scope}[shift={(8,0)}]

\foreach \x in {0,1,2,3}
{
\foreach \y in {0,1,2,3}
{
\fill[gray!50]  (\x-.5,\y)--(\x,\y-.5)--(\x+.5,\y)--(\x,\y+.5);
}}

\foreach \x in {0,1,2,3}
{
\draw (-1,-.5+\x) --++ (4.5-\x,4.5-\x);
\draw (\x-.5,-1) --++ (4.5-\x,4.5-\x);

\draw (-1,3.5-\x) --++ (4.5-\x,-4.5+\x);
\draw (\x-.5,4) --++ (4.5-\x,-4.5+\x);
}

%

\fill[white] (-1,-1) rectangle (-.2,4);
\fill[white] (4,-1) rectangle (3.2,4);
\fill[white] (-1,-1) rectangle (4,-.2);
\fill[white] (-1,4) rectangle (4,3.2);

\end{scope}
\end{tikzpicture}
\caption{From lattice models to graph planar algebras: The vertices, edges and faces in the 2D lattice correspond to shaded regions, crossings and unshaded regions of $D$ in the graph planar algebra $\mathscr{G}_{\bullet}$.} 
\end{figure}
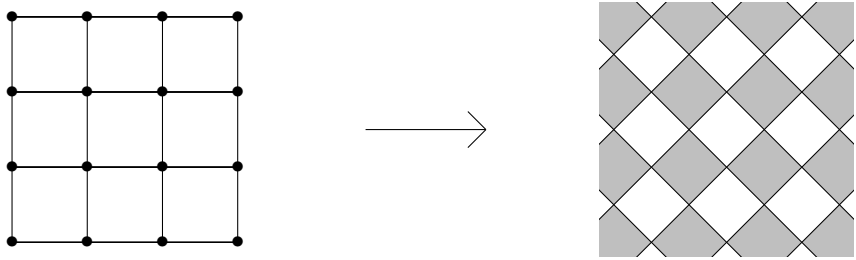

When $T \to 0$, $B(T,J)$ is dominated by 
\raisebox{-.3cm}{
\begin{tikzpicture}
\fill[gray!50] (0,0) rectangle (.6,.8);
\draw (0,0)--++(0,.8);
\draw (.6,0)--++(0,.8);
\end{tikzpicture}}.
In the limit case, the shaded regions are all connected. So all spins in shaded regions are same and the Ising model is ordered.
When $T\to \infty$, $B(T,J)$ is dominated by 
\raisebox{-.3cm}{
\begin{tikzpicture}
\fill[gray!50] (0,0) arc (180:0:.3);
\fill[gray!50] (0,.8) arc (-180:0:.3);
\draw (0,0) arc (180:0:.3);
\draw (0,.8) arc (-180:0:.3);
\end{tikzpicture}}.
In the limit case, the shaded regions are all disconnected. So all spins in shaded regions are independent and the Ising model is disordered.
%

The neighboor interaction $J$ involves the $\Z_2$ symmetry. In planar algebras, that means the 2-box $B(T,J)$ is a flat element with respect to the connection associated with the $\Z_2$ subfactors. 
In general, the flat elements of a graph planar algebra with respect to a connection is a subfactor planar algebra $\mathscr{P}_{\bullet}$  \cite{Ocn88,Jon99}. The 2-box space of $\mathscr{P}_{\bullet}$ is denoted by $\mathscr{P}_{2,\pm}$, where the sign $\pm$ indicates the shading.

People use renormalization groups to study the scaling limit of lattice models, see \cite{MarKad78}.  The idea is rescaling the size of the lattice by combining four vertices of a 2D lattice to one vertex.
Motivated by this idea and the square relation in \cite{JLW}, we introduce the block maps $B_{\lambda}$, $0\leq \lambda\leq 1$, on $\mathscr{P}_{2,\pm}$, which combines four 2-boxes to one. 

For any $x\in\mathscr{P}_{2,\pm}$, we define
\begin{align*}
B_{cm}(x)&=\frac{\delta^2}{\|x\|_1\|x\|_2^2}(\overline{x^*}*x)(\overline{x}*x^*) \;,\\
B_{mc}(x)&=\frac{\delta}{\|x\|_\infty\|x\|_2^2}(\overline{x}\overline{x^*})*(x^*x) \;,\\
B_{\lambda}&=\lambda B_{cm}+(1-\lambda) B_{mc} \;.
\end{align*}
The block maps $B_{\lambda}$ have not been studied for cyclic groups. 
We believe that the asymptotic phenomenon of block maps will shed light on the scaling limit of lattice models. 

We prove that the limit points are multiples of biprojections, a notion introduced by Bisch  for subfactors as a generalization of indicator functions on subgroups \cite{Bis94}.

\begin{theorem}\label{Thm:main1}[Proposition \ref{teq}, Theorem \ref{Thm:RM}]
Suppose $\mathscr{P}_{\bullet}$ is an irreducible subfactor planar algebra.
Then for any $x\in\mathscr{P}_{2,\pm}$ and $B_\lambda$ on $\mathscr{P}_{2,\pm}$, $0\leq \lambda\leq 1$, the sequence $\{B^b_{\lambda}(x)\}_{n\geq 1}$ converges to $0$ or a positive multiple of a biprojection. Moreover, for a non-zero $x$, $B_{\lambda}(x)=x$ if and only if $x$ is a positive multiple of a biprojection.
\end{theorem}
From dynamic system point of view, we consider the set of $x\in\mathscr{P}_{2,\pm}$ whose limits are zeros as a Julia set, and its complement as a Fatou set.
We consider Theorem \ref{Thm:main1} as a 2D quantum phase transition.

\begin{theorem}\label{Thm:mass gap}
For the $\Z_2$ case, 
%
\begin{align}
\lim_{n\to \infty} B^n_{1/2}(B(T,J))=\left\{ 
\begin{aligned}
c_1(\beta ) \JP ~, & \quad \text{~for~}  0< \beta < T_c^{-1} ; \\
0 ~, & \quad   \text{~for~}  \beta = T_c^{-1} ; \\
c_2(\beta) \ID ~, & \quad  \text{~for~}  \beta > T_c^{-1} . 
\end{aligned}
\right.
\end{align}
Moreover, both $c_1(\beta )$ and $c_2(\beta )$ are positive. 
\end{theorem}

The asymptotic phenomenon of block maps coincides with that of the 2D Ising model explained above. It will be interesting to see a relation between the positive scalars $c_1(\beta )$ and $c_2(\beta )$ and the mass gap. 
We also observe that this renormalization procedure approximates Gaussian functions on $\mathbb{R}^n$.

In general, a biprojection can be represented by a Bisch-Jones diagram \raisebox{-.3cm}{
\begin{tikzpicture}
\begin{scope}[xscale=1.5]
\fill[blue!40] (0,0) rectangle (.6,.8);
\begin{scope}[shift={(.15,0)},xscale=.5]
\fill[gray!50] (0,0) arc (180:0:.3);
\fill[gray!50] (0,.8) arc (-180:0:.3);
\draw[] (0,0) arc (180:0:.3);
\draw[] (0,.8) arc (-180:0:.3);
\end{scope}
\draw[] (0,0)--++(0,.8);
\draw[] (.6,0)--++(0,.8);
\end{scope}
\end{tikzpicture}} up to a positive scalar \cite{BisJon97}. A string in the planar diagram is splitted into a parallel pair of strings in the Bisch-Jones diagram. The region between the pair is colored by a third shading. Moreover, the principal graph and the graph planar algebra become refined \cite{MorWal,Liu15}. 

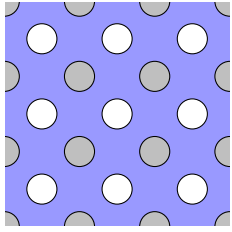
\begin{figure}\label{BJ-partition function}
\begin{tikzpicture}

\fill[blue!40] (0,0) rectangle (3,3);
\foreach \x in {0,1,2,3}
{
\foreach \y in {0,1,2,3}
{
\fill[gray!50] (\x,\y) circle (.2);
\draw (\x,\y) circle (.2);
}}

\foreach \x in {0,1,2,3,4}
{
\foreach \y in {0,1,2,3,4}
{
\fill[white!50] (\x-.5,\y-.5) circle (.2);
\draw (\x-.5,\y-.5) circle (.2);
}}

\fill[white] (-1,-1) rectangle (-.0,4);
\fill[white] (4,-1) rectangle (3.0,4);
\fill[white] (-1,-1) rectangle (4,-.0);
\fill[white] (-1,4) rectangle (4,3.0);
\end{tikzpicture}
\caption{The Bisch-Jones diagrammatic representation of the partition associated with a biprojection.}
\end{figure}

From this point of view, we can represent the planar diagram $D$ by the Bisch-Jones diagram shown in Fig~\ref{BJ-partition function}, when $B(T,J)$ is a biprojection. 
This defines a different limit case of the lattice model, which is partially ordered and partially disordered. 

The proof of theorem \ref{Thm:main1} requires a further development of the Fourier analysis for subfactors. 
The authors  have studied uncertainty principles for subfactors recently in \cite{JLW} and proved Young's inequality and the Hausdorff-Young inequality there. We also introduced bi-shifts of biprojections to characterize the extremizers of various uncertainty principles. For the group case, bi-shifts of biprojections are translations of subcharacters.

In this paper, we characterize the extremal pairs of Young's inequality and extremal operators of the Hausdorff-Young inequality on $\PA_{2,\pm}$. We have not found such characterizations on non-commutative algebras in any literature, even for the representations of a finite group.

\begin{theorem}[Theorem \ref{youngeq}]\label{main2}
Suppose $\mathscr{P}$ is an irreducible subfactor planar algebra.
Let $x,y$ be nonzero in $\mathscr{P}_{2,\pm}$.
Then the following are equivalent:
\begin{itemize}
\item[(1)] $\|x*y\|_r=\frac{1}{\delta}\|x\|_t\|y\|_s$ for some $1<r,t,s<\infty$ such that $\frac{1}{r}+1=\frac{1}{t}+\frac{1}{s}$;
\item[(2)] $\|x*y\|_r=\frac{1}{\delta}\|x\|_t\|y\|_s$ for any $1\leq r,t,s\leq \infty$ such that $\frac{1}{r}+1=\frac{1}{t}+\frac{1}{s}$;
\item[(3)] both $x$ and $y$ are bi-shifts of biprojections, and $\mathcal{R}((\mathfrak{F}^{-1}(x))^*)=\mathcal{R}(\mathfrak{F}^{-1}(y))$;
\item[(4)] there exists a biprojection $B$ in $\mathscr{P}_{2,\pm}$ such that $x=(a_x\ _hB)*\mathfrak{F}(\widetilde{B}_g)$ and $y=\mathfrak{F}(\widetilde{B}_g)*(a_y B_f)$, where $B_g,B_f$ are right shifts of $B$, $_hB$ is a left shift of $B$ and $a_x, a_y$ are elements in $\mathscr{P}_{2,\pm}$ such that $x,y$ are nonzero.
\end{itemize}
\end{theorem}

\begin{theorem}[Theorem \ref{hauseq}]\label{main3}
Suppose $\mathscr{P}$ is an irreducible subfactor planar algebra.
Let $x$ be nonzero in $\mathscr{P}_{2,\pm}$.
Then the following are equivalent:
\begin{itemize}
\item[(1)] $\|\mathfrak{F}(x)\|_{\frac{t}{t-1}}=\left(\frac{1}{\delta}\right)^{\frac{2}{t}-1}\|x\|_t$ for some $1< t<2$.
\item[(2)] $\|\mathfrak{F}(x)\|_{\frac{t}{t-1}}=\left(\frac{1}{\delta}\right)^{\frac{2}{t}-1}\|x\|_t$ for any $1\leq t\leq 2$.
\item[(3)] $x$ is a bi-shift of a biprojection.
\end{itemize}
\end{theorem}

Our proof of the characterization of the extremal pairs in this non-commutative (and non-cocommutative) setting involves the characterizations of extremizers of uncertainty principles and sum set estimates. This is different from previous proofs on commutative algebras, where the corresponding results on Young's inequality, uncertainty principles, and sum set estimates were obtained independently.

The sum set estimate is a new ingredient in subfactor theory from additive combinatorics \cite{TaoVu}.
We prove the sum set estimate and the exact inverse sum set theorem for subfactors:
\begin{theorem}[Theorem \ref{lowerb}, Proposition \ref{upperb}, Theorem \ref{lowerm}, Proposition \ref{1pyoung}]\label{main1}
Suppose $\mathscr{P}$ is an irreducible subfactor planar algebra.
Let $p,q$ be projections in $\mathscr{P}_{2,\pm}$.
Then
\be \label{Sumsetestimate1}
\max\{tr_2(p),tr_2(q)\}\leq \mathcal{S}(p*q)\leq tr_2(p)tr_2(q),
\ee
where $\mathcal{S}(x)=tr_2(\mathcal{R}(x))$ and $\mathcal{R}(x)$ is the range projection of $x$ in $\mathscr{P}_{2,\pm}$.
Moreover, the following statements are equivalent:
\begin{itemize}
\item[(1)] $\mathcal{S}(p*q)=tr_2(p)$;
\item[(2)] $\frac{\delta}{tr_2(q)}p*q$ is a projection;
\item[(3)] $\mathcal{S}(p*(q*\overline{q})^{*(m)})*q^{*(j)})=tr_2(p)$, for some $m\geq 0$, $j\in\{0,1\}$, $m+j>0$, where $q^{*(0)}=e_1$, where $e_1$ is the Jones projection;
\item[(4)] $\mathcal{S}(p*(q*\overline{q})^{*(m)})*q^{*(j)})=tr_2(p)$, for any $m\geq 0$, $j\in\{0,1\}$, $m+j>0$;
\item[(5)] there exists a biprojection $B$ in $\mathscr{P}_{2,\pm}$ such that $q$ is a right subshift of $B$ and $p=\mathcal{R}(x*B)$ for some $x>0$;
\item[(6)] $\|p*q\|_t=\frac{1}{\delta}\|p\|_t\|q\|_1$ for some $1<t<\infty$;
\item[(7)] $\|p*q\|_t=\frac{1}{\delta}\|p\|_t\|q\|_1$ for any $1\leq t\leq \infty$;
\item[(8)] $B_1\leq B_2$,
\end{itemize}
where $B_1$ be the biprojection generated by $q*\overline{q}$ and $B_2$ be the spectral projection of $\overline{p}*p$ corresponding to $\frac{tr_2(p)}{\delta}$.
\end{theorem}

The sum set estimate \eqref{Sumsetestimate1} is obvious on the set of group elements, but it is non-trivial for subfactors, even for the representations of a finite group, see Corollary \ref{Cor:sumset}.
One obtains interesting results while applying these general results on subfactors to particular examples, see \S \ref{Sec:Concluding Remarks} for a short discussion.

The paper is organized as follows.
In \S \ref{Sec:Preliminary}, we recall some notations for subfactors and related results on its Fourier analysis.
In \S \ref{ce}, we study the convolution equation $a*b=a$ and characterize its solutions.
In \S \ref{ssumset}, we prove a sum set estimate and the exact inverse sum set theorem for subfactors.
In \S \ref{Sec:Young}, we characterize the extremal pairs of Young's inequality for subfactor.
In \S \ref{Sec:Hausdorff-Young}, we characterize the extremal operators of the Hausdorff-Young inequality for subfactors.
In \S \ref{Sec:Block Maps}, we introduce the block maps for subfactors and study their dynamic systems.
In \S \ref{Sec:Concluding Remarks}, we discuss the difference between Fourier analysis on subfactors and commutative algebras. We summarize the characterizations of bi-shifts of biprojections in Theorem \ref{Thm:bishifts}.

\begin{ac}
Zhengwei Liu would like to thank V. F. R. Jones for constant encouragement and T. Tao for helpful suggestions.
Parts of the work was done during visits of Zhengwei Liu and Jinsong Wu to Hebei Normal University.
Chunlan Jiang was supported in part by NSFC (Grant no. A010602) and Foundation for the Author of National
Excellent Doctoral Dissertation of China (Grant no. 201116).
Zhengwei Liu was supported by a grant from Templeton Religion Trust and an AMS-Simons Travel Grant.
Jinsong Wu was supported by NSFC (Grant no. A010602) and partially supported by ``PCSIRT" and the Fundamental Research Funds for the Central Universities (WK001000004).
\end{ac}

\section{Preliminaries}\label{Sec:Preliminary}
\subsection{Subfactors and the dual pair of $C^*$-algebras}
Jones classified the Jones index $\delta^2$ of a subfactor in \cite{Jon83}. It generalizes the order of a group and can be non-integers.
Each subfactor defines a pair of finite dimensional $C^*$-algebras which generalize a finite group and its dual, namely the representation category of the group.
This pair are related by the Fourier transform which was introduced by Ocneanu for subfactors \cite{Ocn88}.


Given a finite index type II$_1$ subfactor $\N\subset \M$, one obtains an $\N-\N$ bimodule $L^2(\M)$. The bimodule maps $\hom_{\N-\N}(L^2(M))$ forms a $C^*$-algebra. On the other hand, the multiplication on $M$ defines a bounded $\N-\N$ bimodule map $\gamma$ from the Connes fusion $L^2(\M)\otimes_{\N} L^2(\M)$ to $L^2(\M)$. The associativity of the multiplication on $M$ tells that $\gamma$ is a Frobenius algebra in the $\N-\N$ bimodule category. For $x,y \in \hom_{\N-\N}(L^2(M))$, their convolution is defined as $\gamma (x\otimes y) \gamma^*$. Then $\hom_{\N-\N}(L^2(M))$ forms another $C^*$-algebra, where the involution is defined by the modular conjugation in Tomita-Takesaki theory. The identity map on the space $\hom_{\N-\N}(L^2(M))$ from one $C^*$-algebra to the other plays the role of the Fourier transform.
The projection from $L^2(\M)$ to $L^2(\N)$ is a bimodule map, called the Jones projection. The Fourier transform switches the identity and the Jones projection in the pair of $C^*$ algebras.
Furthermore, both $C^*$-algebras are equipped with a Markov trace which can be defined as the pull back of the delta function on the Jones projection by the Fourier transform.

In the algebraic framework, this can also be formalized by a Frobenius algebra $\gamma$ in a rigid $C^*$-tensor category in the same way \cite{Mug03}. Then the target space $\hom_{\N-\N}(L^2(M))$ becomes $\hom(\gamma,\gamma)$.

When $\M=\N\rtimes G$, for an outer action of a finite group $G$, $\hom_{\N-\N}(L^2(\M))\cong L^2(G)$.
The Jones index of $\N\subset \M$ is the order of the group $G$.
The convolution coincides with the usual convolution on $L^2(G)$. The Markov trace is defined by the discrete measure on $G$. The $C^*$-algebra defined by the convolution is the group algebra acting on the left regular representation of $G$. The Markov trace is defined by trace of the matrix.

In general, this pair of $C^*$-algebras are captured by the 2-box space of the planar algebra $\mathscr{P}=\{\mathscr{P}_{n,\pm}\}_{n\geq 0}$ of the subfactor \cite{Jon99}, where readers can find the definition and examples.

\begin{notation}
In this paper, the $\$$ signs are always on the left side of discs of planar tangles.
We omit the output disc and the $\$$ signs.
\end{notation}

In the planar algebra framework, the space $\hom_{\N-\N}(L^2(M))$ is the 2-box space $\PA_{2,+}$. Its element $x$ is represented by a shaded diagram with four boundary points:
\raisebox{-.3cm}{
\begin{tikzpicture}
  \fill[gray!50] (.1,-.2) rectangle (.4,.7);
  \draw (.1,-.2)--++(0,.9);
  \draw (.4,-.2)--++(0,.9);
  \fill[white] (0,0) rectangle (.5,.5);
  \draw(0,0) rectangle (.5,.5);
  \node at (.25,.25) {$x$};
\end{tikzpicture}}.
The identity is represented by
\raisebox{-.3cm}{
\begin{tikzpicture}
  \fill[gray!50] (.1,-.2) rectangle (.4,.7);
  \draw (.1,-.2)--++(0,.9);
  \draw (.4,-.2)--++(0,.9);
\end{tikzpicture}}.
The multiplication $xy$ is represented by
\raisebox{-.6cm}{
\begin{tikzpicture}
  \fill[gray!50] (.1,-.2) rectangle (.4,1.4);
  \draw (.1,-.2)--++(0,1.6);
  \draw (.4,-.2)--++(0,1.6);
  \fill[white] (0,0) rectangle (.5,.5);
  \draw(0,0) rectangle (.5,.5);
  \node at (.25,.25) {$x$};
  \fill[white] (0,.7) rectangle (.5,1.2);
  \draw(0,.7) rectangle (.5,1.2);
  \node at (.25,.95) {$y$};
\end{tikzpicture}}.
The coproduct $x*y$ is represented by
\raisebox{-.3cm}{
\begin{tikzpicture}
  \fill[gray!50] (.1,-.2) rectangle (1.1,.7);
  \draw (.1,-.2)--++(0,.9);
  \draw (1.1,-.2)--++(0,.9);
  \fill[white] (.4,0) to [bend left=-30] (.8,0)--(.8,.5) to [bend left=-30] (.4,.5);
  \draw (.4,0) to [bend left=-30] (.8,0)--(.8,.5) to [bend left=-30] (.4,.5);
  \fill[white] (0,0) rectangle (.5,.5);
  \draw(0,0) rectangle (.5,.5);
  \node at (.25,.25) {$x$};
  \fill[white] (.7,0) rectangle (1.2,.5);
  \draw(.7,0) rectangle (1.2,.5);
  \node at (.95,.25) {$y$};
\end{tikzpicture}}.
For the group case, $\delta x*y$ is the convolution of functions in $L^2(G)$.
The Markov trace $tr_2(x)$ is represented by
\raisebox{-.6cm}{
\begin{tikzpicture}
  \fill[gray!50] (.1,0)--++(0,.5) arc (180:0:.5)--++(0,-.5) arc (0:-180:.5);
  \draw (.1,0)--++(0,.5) arc (180:0:.5)--++(0,-.5) arc (0:-180:.5);
  \fill[white] (.4,0)--++(0,.5) arc (180:0:.2)--++(0,-.5) arc (0:-180:.2);
  \draw (.4,0)--++(0,.5) arc (180:0:.2)--++(0,-.5) arc (0:-180:.2);
  \fill[white] (0,0) rectangle (.5,.5);
  \draw(0,0) rectangle (.5,.5);
  \node at (.25,.25) {$x$};
\end{tikzpicture}}.
The adjoint operation is represented by a vertical reflection.

The elements in $\PA_{2,-}$ are represented diagrams with the opposite shading.
The string Fourier transform (SFT) $\mathfrak{F}$ from $\mathscr{P}_{2,\pm}$ onto $\mathscr{P}_{2,\mp}$ is the clockwise $1$-string rotation, or a $90^\circ$ rotation geometrically.
For any $x,y$ in $\mathscr{P}_{2,\pm}$, we have $(\mathfrak{F}(x))^*=\mathfrak{F}^{-1}(x^*)$,
and
$$\F(xy)=\F(x)*\F(y).$$
We denote by $\overline{x}:=\F^2(x)$ the contragredient of $x$

An advantage of planar algebras is that we can study this pair of $C^*$-algebras in a bigger space $\PA_{n,\pm}$ with compatible topological properties.
The tensor product $x\otimes y$ is represented by
\raisebox{-.3cm}{
\begin{tikzpicture}
  \fill[gray!50] (.1,-.2) rectangle (.4,.7);
  \draw (.1,-.2)--++(0,.9);
  \draw (.4,-.2)--++(0,.9);
  \fill[white] (0,0) rectangle (.5,.5);
  \draw(0,0) rectangle (.5,.5);
  \node at (.25,.25) {$x$};
  \begin{scope}[shift={(.7,0)}]
  \fill[gray!50] (.1,-.2) rectangle (.4,.7);
  \draw (.1,-.2)--++(0,.9);
  \draw (.4,-.2)--++(0,.9);
  \fill[white] (0,0) rectangle (.5,.5);
  \draw(0,0) rectangle (.5,.5);
  \node at (.25,.25) {$y$};
  \end{scope}
\end{tikzpicture}}.

We will ignore the alternating shading to simplify the pictures in the rest of the paper.

\subsection{Previous results}
We briefly recall some notations and results in \cite{Bis94,Liu16,JLW}.
Suppose $\mathscr{P}=\{\mathscr{P}_{n,\pm}\}_{n\geq 0}$ is a subfactor planar algebra.
For any $x\in \mathscr{P}_{2,\pm}$, we denote by $w_x|x|$ the polar decomposition of $x$, $\mathcal{R}(x)$ the range projection of $x$, and $\mathcal{S}(x)=tr_2(\mathcal{R}(x))$.
We say $x\sim y$ if $\mathcal{R}(x)=\mathcal{R}(y)$ and $x\preceq y$ if $\mathcal{R}(x)\leq \mathcal{R}(y)$.

A projection $B$ in $\mathscr{P}_{2,\pm}$ is called a biprojection if $\mathfrak{F}(B)$ is a multiple of a projection \cite{Bis94}. It generalizes the indicator function on subgroups of a finite group.

A biprojection $B$ generated by an element $x$ in $\mathscr{P}_{2,\pm}$ is the smallest biprojection satisfying $BxB=x$ \cite{Liu16}, where the existence of the smallest one is proved.

\begin{notation}
For a positive operator $x$ in $\mathscr{P}_{2,\pm}$,
we define $B_1(x)$ to be the biprojection generated $x*\overline{x}$.
\end{notation}

We introduced bi-shifts of biprojections in \cite{JLW} to generalize the translations of subcharacters on finite abelian groups.
A projection $p$ in $\mathscr{P}_{2,\pm}$ is called a left (or right) shift of a biprojection $B$, if $tr_2(p)=tr_2(B)$ and $p*B=\frac{tr_2(B)}{\delta}p$ (or $B*p=\frac{tr_2(B)}{\delta}p$).
Let $\tilde{B}$ be the range projection of $\mathfrak{F}(B)$ in $\mathscr{P}_{2,\mp}$.
A nonzero element $x$ in $\mathscr{P}_{2,\pm}$ is a bi-shift of the biprojection $B$ if there exists a right shift $B_g$ of $B$ and a right shift $\tilde{B}_h$ of $\tilde{B}$ and an element $y$ in $\mathscr{P}_{2,\pm}$ such that $x=\mathfrak{F}(\tilde{B}_h)*(yB_g)$.
That means the range projections of $x$ and $\F(x)$ are shifts of $B$ and $\tilde{B}$ respectively. The uniqueness of such element for given range projections is proved by the Hardy uncertainty principle \cite{JLW}.

\begin{theorem}\label{Thm:Schur}[Schur Product Theorem]
Suppose that $\mathscr{P}$ is a subfactor planar algebra and $a,b\in \mathscr{P}_{2,\pm}$ are positive.
Then $a*b>0$.
\end{theorem}
\begin{proof}
This is Theorem 4.1 in \cite{Liu16}.
\end{proof}

\begin{lemma}\label{permu}
Suppose $\mathscr{P}$ is a subfactor planar algebra.
Let $a,b,c$ be in $\mathscr{P}_{2,\pm}$.
Then
\begin{eqnarray*}
&&tr_2((a*b)\overline{c})=tr_2((b*c)\overline{a})=tr_2((c*a)\overline{b})\\
&=& tr_2((\overline{c}*\overline{b})a)=tr_2((\overline{a}*\overline{c})b)=tr_2((\overline{b}*\overline{a})c)
\end{eqnarray*}
\end{lemma}
\begin{proof}
This is the Lemma 4.6 in \cite{Liu16}, Lemma 3.4 in \cite{JLW}.
\end{proof}

\begin{lemma}\label{ran}
Suppose $\mathscr{P}$ is a subfactor planar algebra.
Let $x,y$ be in $\mathscr{P}_{2,\pm}$.
Then
$$\mathcal{R}(x*y)\leq \mathcal{R}(\mathcal{R}(x)*\mathcal{R}(y)).$$
\end{lemma}
\begin{proof}
This is the Lemma 3.5 in \cite{JLW}.
\end{proof}

\begin{proposition}[H\"{o}lder's Inequality]\label{holder}
Suppose $\mathscr{P}$ is a subfactor planar algebra.
Let $a,b,c$ be in $\mathscr{P}_{2,\pm}$.
Then
$$|tr_2(ab)|\leq \|a\|_p\|b\|_q,$$
where $1\leq p \leq \infty$, $\frac{1}{p}+\frac{1}{q}=1$.
Moreover $|tr_2(x^*y)|=\|x\|_p\|y\|_q$ if and only if
$$x=u|x|, y=\lambda u|y|,\frac{|x|^p}{\|x\|_p^p}=\frac{|y|^q}{\|y\|_q^q},$$
for some unitary element $u$ and some complex number $\lambda$ with $|\lambda|=1$.
\end{proposition}
\begin{proof}
The proof can be found in \cite{Xu}.
See also Proposition 4.3 and 4.5 in \cite{JLW}.
\end{proof}

\begin{lemma}\label{preceq}
Suppose $\mathscr{P}$ is a subfactor planar algebra and $a,b,c,d\in\mathscr{P}_{2,\pm}$ are positive.
If $a\preceq c$ and $b\preceq d$, then $a*b\preceq c*d$.
\end{lemma}
\begin{proof}
This is Lemma 4.8 in \cite{Liu16}
\end{proof}

\begin{corollary}
Suppose $\mathscr{P}$ is a subfactor planar algebra and $x\in\mathscr{P}_{2,\pm}$ is positive.
Then $B_1(x)=B_1(\mathcal{R}(x))$.
\end{corollary}

\begin{proposition}[the Hausdorff-Young Inequality]\label{hausyoung}
Suppose $\mathscr{P}$ is an irreducible subfactor planar algebra.
Let $x$ be in $\mathscr{P}_{2,\pm}$. Then
$$\|\mathfrak{F}(x)\|_t\leq \left(\frac{1}{\delta}\right)^{1-\frac{2}{t}}\|x\|_s,$$
where $2\leq t\leq \infty$ and $\frac{1}{t}+\frac{1}{s}=1$.
\end{proposition}
\begin{proof}
This is Theorem 4.8 in \cite{JLW}.
\end{proof}

\begin{proposition}[Young's Inequality]\label{young}
Suppose $\mathscr{P}$ is an irreducible subfactor planar algebra.
Let $x,y$ be in $\mathscr{P}_{2,\pm}$.
Then
$$\|x*y\|_r\leq \frac{1}{\delta}\|x\|_t\|y\|_s,$$
where $1\leq t,s,r \leq \infty$, $\frac{1}{r}+1=\frac{1}{t}+\frac{1}{s}.$
\end{proposition}
\begin{proof}
This is Theorem 4.13 in \cite{JLW}.
\end{proof}

\begin{proposition}\label{spec}
Suppose $\mathscr{P}$ is an irreducible subfactor planar algebra and $x\in\mathscr{P}_{2,\pm}$.
If $\mathfrak{F}^{-1}(x)$ is extremal, then $xB$ is an extremal bi-partial isometry, where $B$ is the spectral projection of $|x|$ with spectrum $\|x\|_\infty$.
\end{proposition}
\begin{proof}
This is Corollary 6.12 in \cite{JLW}.
\end{proof}

\begin{notation}
Suppose $x$ is a positive operator in $\mathscr{P}_{2,\pm}$.
Then $\mathfrak{F}(\overline{x}*x)>0$ and
$$\|\overline{x}*x\|_\infty=\|\mathfrak{F}(\overline{x}*x)\|_1=\frac{\|x\|_2^2}{\delta}.$$
We denote by $B_2(x)$ to be the spectrum projection of $\overline{x}*x$ with spectrum $\frac{\|x\|_2^2}{\delta}$.
By Proposition \ref{spec}, $B_2$ is a biprojection.
\end{notation}

\section{Convolution Equations}\label{ce}
In this section, we study the convolution equation $a*b=a$ for $a,b\in\mathscr{P}_{2,\pm}$, where $\mathscr{P}$ is a subfactor planar algebra.
We begin with an idempotent theorem for subfactor planar algebras.

\begin{proposition}\label{converge}
Suppose $\mathscr{P}$ is a subfactor planar algebra.
Let $x\in\mathscr{P}_{2,\pm}$ be positive such that $tr_2(x)=\delta$.
Then the Ces\`{a}ro mean
$$x_n=\frac{1}{n}\sum_{k=1}^n x^{*(k)}$$
converges to an element $a$ in $\mathscr{P}_{2,\pm}$ such that $a*a=a$, where $x^{*(k)}=\underset{k}{\underbrace{x*\cdots*x}}$.
\end{proposition}
\begin{proof}
By Proposition \ref{young}(Young's inequality), we have that $\|x_n\|_1\leq \delta$.
Hence $\{x_n\}_{n\geq 1}$ is a compact set.
Let $a$ be an accumulation point for $\{x_n\}_{n\geq 1}$.
Choosing a sequence $(n_k)_{k\geq 0}$ such that $x_{n_k}\to a$ in $\|\cdot\|_1$-topology, we get for any $\epsilon>0$, there exists $k_0$ and $N$ such that $n_k>N>\frac{1}{\epsilon}$, $k\geq k_0$, $\|x_{n_k}-a\|_1<\epsilon$,
\begin{eqnarray*}
\|x*a-a\|_1&\leq &\|x*x_{n_k}-x_{n_k}\|_1+\|x*(a-x_{n_k})\|_1+\|x_{n_k}-a\|_1\\
&\leq &\frac{1}{N}\|x-x^{*(n_k+1)}\|_1+\epsilon+\epsilon\leq 2(\delta+1)\epsilon.
\end{eqnarray*}
We have $x*a=a$.
Similarly, $a*x=a$.
Moreover $x_n*a=a*x_n=a$ for any $n$.
If $a'$ is another accumulation point, we have $a'*a=a*a'=a$.
Symmetrically, $a*a'=a'*a=a'$.
Therefore $a=a'$ and $a*a=a$.
\end{proof}

\begin{proposition}\label{idem}
Suppose $\mathscr{P}$ is an irreducible subfactor planar algebra.
Let $a\in\mathscr{P}_{2,\pm}$ be a nonzero positive element such that $a*a=a$.
Then $\frac{1}{tr_2(\mathfrak{F}(a))}a$ is a biprojection.
\end{proposition}
\begin{proof}
Since $tr_2(a*a)=tr_2(a)$, we have $tr_2(a)=\delta$.
By Proposition \ref{hausyoung} (the Hausdorff-Young inequality), we have that $\|\mathfrak{F}(a)\|_\infty\leq \frac{\|a\|_1}{\delta}=1$. Taking the SFT, we obtain that $\mathfrak{F}(a)^2=\mathfrak{F}(a)$.
Hence $\mathfrak{F}(a)$ is a contractive idempotent, i.e. $\mathfrak{F}(a)$ is a projection.
By Proposition \ref{spec}, we obtain that $\mathfrak{F}(a)$ is a biprojection and so is $\frac{1}{tr_2(\mathfrak{F}(a))}a$.
\end{proof}

\begin{remark}
Proposition \ref{idem} is a consequence of Theorem 4.21 in \cite{Liu16}.
\end{remark}

Combining the proofs of Proposition \ref{converge}, \ref{idem} above, we have a slightly general version.
Recall that an element $x$ in $\mathscr{P}_{2,\pm}$ is extremal if $\|\mathfrak{F}(x)\|_\infty=\frac{1}{\delta}\|x\|_1$.

\begin{theorem}\label{idemp}
Suppose $\mathscr{P}$ is an irreducible subfactor planar algebra.
Let $x\in\mathscr{P}_{2,\pm}$ be such that $\|x\|_1\leq \delta$.
Then $x_n=\frac{1}{n}\sum_{k=1}^n x^{*(k)}$ converges to $0$ or a bi-shift $w$ of a biprojection $B$ in $\mathscr{P}_{2,\pm}$ such that $\mathfrak{F}(w)$ is a right shift of $B$.
\end{theorem}

\begin{definition}
Suppose $\mathscr{P}$ is an irreducible subfactor planar algebra
and $p$ is a projection in $\mathscr{P}_{2,\pm}$.
We say a positive operator $x$ in $\mathscr{P}_{2,\pm}$ is (left- or) right-absorbed by $p$ if ($x*p\sim p$ or) $p*x\sim p$.
\end{definition}

\begin{proposition}\label{Prop:RAS}
For a projection $p$ in $\mathscr{P}_{2,\pm}$, there is a maximal projection $B$ right-absorbed by $p$.
Moreover, $B$ is a biprojection.
\end{proposition}

\begin{proof}
Note that the Jones projection is right-absorbed by any projection $p$.
If $a,b$ are right-absorbed by a projection $p$, then $a*b$ is right-absorbed by $p$.
Therefore there is a maximal projection $B$ right-absorbed by $p$.
Moreover, we have that $B*B\sim B$, thus $B$ is a biprojection by Theorem 4.12 in \cite{Liu16}.
\end{proof}

\begin{definition}
We call the biprojection $B$ in Proposition \ref{Prop:RAS} the right-absorbing-support (RAS) of $p$.
Similarly, we define the left-absorbing-support (LAS) of $p$.
\end{definition}

\begin{proposition}\label{idem2}
Suppose $\mathscr{P}$ is an irreducible subfactor planar algebra.
Let $a,b\in \mathscr{P}_{2,\pm}$ be nonzero positive elements such that $a*b=\frac{tr_2(b)}{\delta}a$, and $tr_2(b)\neq 0$.
Then $a*\overline{b}=\frac{tr_2(b)}{\delta}a$, and
$$a*(xb)=a*(bx)=\frac{tr_2(xb)}{\delta}a$$
for any $x\in\mathscr{P}_{2,\pm}$.
\end{proposition}


\newcommand{\Ga}[2]
{
\begin{scope}[shift={(#1)}]
\draw (0,0)--+(0,1/2);
\draw (1/2,0)--+(0,1/2);
\draw (-1/6,1/2) rectangle +(5/6,5/6);
\node at (1/4,1/2+5/12) {$#2$};
\end{scope}
}

\newcommand{\Gb}[2]
{
\Ga{1,0}{#2}
\draw (1/2,0)--(1/2,1/2+5/6) arc (180:0:1/4);
\draw (0,0)--(0,1/2+5/6)--+(0,1/2);
\draw (0,11/6+5/6)--+(0,1/2);
\draw (3/2,1/2+5/6)--+(0,1/2);
\draw (3/2,11/6+5/6)--+(0,1/2);
\draw (-1/6,11/6) rectangle (3/2+1/6,11/6+5/6);
\node at (3/4,11/6+5/12) {$#1$};
}

\newcommand{\Gc}[2]
{
\Ga{0,0}{#1}
\Ga{1,0}{#2}
\draw (1/2,1/2+5/6) arc (180:0:1/4);
\draw (0,1/2+5/6)--+(0,1/2);
\draw (3/2,1/2+5/6)--+(0,1/2);
}

\newcommand{\Gcob}
{
\raisebox{-1.2cm}{\begin{tikzpicture}
\Gb{a}{\overline{b}^{\frac{1}{2}}}
\end{tikzpicture}}
}

\newcommand{\Gtea}
{
\raisebox{-.7cm}{
\begin{tikzpicture}
\Gc{a}{\overline{b}^{\frac{1}{2}}}
\end{tikzpicture}}
}

\begin{proof}
Since $a*b=a$, we have $tr_2(b)=\delta$.
First, we would like to show that

Take $x=\Gcob$ and $y=\Gtea$. To prove $x=y$, we show that
\be\label{Equ0}
tr_2((x-y)^*(x-y))=0.
\ee

Expanding the left hand side, we have that
\begin{align*}
&tr_2((x-y)^*(x-y)) \\
=&
tr_2\left(
\raisebox{-3cm}{
\begin{tikzpicture}
\Gb{a}{\overline{b}^{\frac{1}{2}}}
\begin{scope}[shift={(0,1/2)},yscale=-1,xscale=1]
\Gb{a}{\overline{b}^{\frac{1}{2}}}
\end{scope}
\end{tikzpicture}}
-
\raisebox{-2.2cm}{
\begin{tikzpicture}
\Gb{a}{\overline{b}^{\frac{1}{2}}}
\begin{scope}[shift={(0,1/2)},yscale=-1,xscale=1]
\Gc{a}{\overline{b}^{\frac{1}{2}}}
\end{scope}
\end{tikzpicture}}
-
\raisebox{-2.2cm}{
\begin{tikzpicture}
\Gc{a}{\overline{b}^{\frac{1}{2}}}
\begin{scope}[shift={(0,1/2)},yscale=-1,xscale=1]
\Gb{a}{\overline{b}^{\frac{1}{2}}}
\end{scope}
\end{tikzpicture}}
+
\raisebox{-1.6cm}{
\begin{tikzpicture}
\Gc{a}{\overline{b}^{\frac{1}{2}}}
\begin{scope}[shift={(0,1/2)},yscale=-1,xscale=1]
\Gc{a}{\overline{b}^{\frac{1}{2}}}
\end{scope}
\end{tikzpicture}}
\right)\\
=&(I)-(II)-(III)+(IV).
\end{align*}
Now we have that $(I)=tr_2(a^2)\frac{tr_2(b)}{\delta}=tr_2(a^2)$
and  $(II)=tr_2((a*\overline{b})a)$.
By Lemma \ref{permu}, we have
$$tr_2((a*\overline{b})a)=tr_2((a*b)a)=tr_2(a^2).$$
Note that $(III)=(II)$ and $(IV)=tr_2(a^2)$.
Hence Equation \eqref{Equ0} is true.
Moreover, we have
\begin{equation}\label{id3}
\raisebox{-1.2cm}{
\begin{tikzpicture}
\Gb{a}{\overline{b}}
\end{tikzpicture}}
=
\raisebox{-.7cm}{
\begin{tikzpicture}
\Gc{a}{\overline{b}}
\end{tikzpicture}}
\end{equation}
Adding caps to the middle bottom of the diagram above, we obtain that
$$a=a*\overline{b}.$$
Hence by the argument above again for $a*\overline{b}=a$, we have
$$
\raisebox{-1.2cm}{
\begin{tikzpicture}
\Gb{a}{b}
\end{tikzpicture}}
=
\raisebox{-.7cm}{
\begin{tikzpicture}
\Gc{a}{b}
\end{tikzpicture}}
$$
Multiplying by $x$ on the bottom of $b$ and taking a cap to the middle of the bottom, we will get that
$$a tr_2(xb)=a*(xb).$$
Similarly, by using Equation \eqref{id3}, we have that
$$a tr_2(xb)=a*(bx).$$
Therefore, $$a*(xb)=a*(bx)=\frac{tr_2(xb)}{\delta}a.$$
\end{proof}

\begin{remark}
The proof here can give an alternative proof of Proposition \ref{idem}.
\end{remark}

\begin{remark}
In Proposition \ref{idem2}, the $a$ can be any element in $\mathscr{P}_{2,\pm}$ such that $tr_2(a)\neq 0$.
\end{remark}

\begin{proposition}\label{alg}
Suppose $\mathscr{P}$ is an irreducible subfactor planar algebra.
Let $B$ be a biprojection in $\mathscr{P}_{2,\pm}$.
Then $\{a\in\mathscr{P}_{2,\pm}|a*(\frac{\delta}{tr_2(B)}B)=a\}=\mathfrak{A}$ is a C$^*$-algebra.
\end{proposition}

\begin{remark}
If we consider the biprojection $B$ as a Bisch-Jones diagram \cite{BisJon97} (up to a scalar),
\raisebox{-.4cm}{
\begin{tikzpicture}
\draw (0,0)--+(0,1);
\draw (1,0)--+(0,1);
\draw (1/3,0)--(1/3,1/6) arc (180:0:1/6)--(2/3,0);
\draw (1/3,1)--(1/3,5/6) arc (-180:0:1/6)--(2/3,1);
\end{tikzpicture}},
then one can represent $a$ as a Bisch-Jones diagram of the following form,
\raisebox{-.4cm}{
\begin{tikzpicture}
\draw (-1/6,1/3) rectangle (5/6,2/3);
\foreach \x in {0,1/3,2/3}
{
\draw (\x,0)--+(0,1/3);
\draw (\x,2/3)--+(0,1/3);
}
\draw (1,0)--(1,1);
\end{tikzpicture}}. We see that all such $a$'s form a $C^*$-algebra. Here we give an algebraic proof, which maybe useful for general case.
\end{remark}

\begin{proof}
Let $a,b\in\mathscr{P}_{2,\pm}$ be such that
$$a*(\frac{\delta}{tr_2(B)}B)=a, \quad b*(\frac{\delta}{tr_2(B)}B)=b.$$
Then
$$a^**(\frac{\delta}{tr_2(B)}B)=a^*,\quad (a+b)*(\frac{\delta}{tr_2(B)}B)=a+b,\quad \lambda a*(\frac{\delta}{tr_2(B)}B)=\lambda a$$
for any $\lambda\in \mathbb{C}$.

Now we will show that $(ab)*(\frac{\delta}{tr_2(B)}B)=ab$.
Since $a*(\frac{\delta}{tr_2(B)}B)=a$, by taking the SFT, we have
$\mathfrak{F}^{-1}(a)\mathfrak{F}^{-1}(B)\frac{\delta}{tr_2(B)}=\mathfrak{F}^{-1}(a)$.
Let $\tilde{B}$ be the range projection of $\mathfrak{F}^{-1}(B)$.
Then $\mathcal{R}(\mathfrak{F}^{-1}(a)^*)\leq \tilde{B}$.
Similarly, $\mathcal{R}(\mathfrak{F}^{-1}(b)^*)\leq \tilde{B}$.
By Lemma \ref{ran}, we have that
\begin{eqnarray*}
\mathcal{R}(\mathfrak{F}^{-1}(b)^**\mathfrak{F}^{-1}(a)^*)
&\leq& \mathcal{R}(\mathcal{R}(\mathfrak{F}^{-1}(b)^*)*\mathcal{R}(\mathfrak{F}^{-1}(a))^*)\\
&\leq & \mathcal{R}(\tilde{B}*\tilde{B})=\tilde{B}.
\end{eqnarray*}
On the other hand
\begin{eqnarray*}
\mathfrak{F}^{-1}(b)^**\mathfrak{F}^{-1}(a)^*&=&\mathfrak{F}(b^*)*\mathfrak{F}(a^*)\\
&=&\mathfrak{F}(b^*a^*)=\mathfrak{F}^{-1}(ab)^*.
\end{eqnarray*}
We see that $\mathfrak{F}^{-1}(ab)\mathfrak{F}^{-1}(B)\frac{\delta}{tr_2(B)}=\mathfrak{F}^{-1}(ab)$, i.e. $(ab)*(\frac{\delta}{tr_2(B)}B)=ab$.
\end{proof}

\section{The Exact Inverse Sum Set Theorem}\label{ssumset}
In this section, we prove a sum set estimate and the exact inverse sum set theorem for subfactor planar algebras.
First let us recall the results on finite abelian groups which have been well-studied in additive combinatorics \cite{TaoVu}.

Let $A, B$ be additive sets with common ambient finite additive group $G$.
A fundamental problem in additive combinatorics is the inverse sum set problem: if $A+B$ or $A-B$ is small, what can one say about $A$ and $B$?
The sum set estimates are given by
\be \label{Sumsetestimate}
\max\{|A|,|B|\}\leq |A+B|,|A-B|\leq |A||B|,
\ee
where $|A|$ is the cardinality of $A$.
The exact inverse sum set theorem \cite{TaoVu} says that the following statements are equivalent:
\begin{itemize}
\item[(1)] $|A+B|=|A|$;
\item[(2)] $|A-B|=|A|$;
\item[(3)] $|A-nB-mB|=|A|$ for at least one pair of integers $(n,m)\neq (0,0)$;
\item[(4)] $|A-nB-mB|=|A|$ for all integers $n,m$;
\item[(5)] there exists a finite subgroup $H$ of $G$ such that $B$ is contained in a coset of $H$, and $A$ is a union of cosets of $H$.
\end{itemize}

The sum set estimate is obvious in the group case, since the convolution with a group element preserves the cardinality of the set.
However, it is non-trivial on subfactor planar algebras.
In subfactor planar algebras, a 2-box projection is a group element if and only if it has trace one.
Usually the trace of a minimal projection is greater than 1. The original proof in additive combinatorics does not apply to subfactors. We give a new proof of these results using Young's inequalities and uncertainty principles.
We apply these results to characterize the extremal pairs of Young's inequalities in \S \ref{Sec:Young}.

\begin{theorem}\label{lowerb}[Sum set estimate]
Suppose $\mathscr{P}$ is an irreducible subfactor planar algebra. Let $p$, $q$ be projections in $\mathscr{P}_{2,\pm}$. Then
$$\max\{tr_2(p),tr_2(q)\}\leq \mathcal{S}(p*q).$$
Moreover, $\mathcal{S}(p*q)=tr_2(q)$ if and only if $\frac{\delta}{tr_2(p)}p*q$ is a projection
\end{theorem}
\begin{proof}
From Proposition \ref{holder} (H\"{o}lder's inequality), we obtain
$$\|p*q\|_1\leq \|\mathcal{R}(p*q)\|_2\|p*q\|_2.$$
Note that
$$\|p*q\|_1=tr_2(p*q)=\frac{tr_2(p)tr_2(q)}{\delta},$$
$$\|\mathcal{R}(p*q)\|_2^2=\mathcal{S}(p*q).$$
By Proposition \ref{young} (Young's inequality), we have
$$\|p*q\|_2\leq \frac{\|p\|_1\|q\|_2}{\delta}=\frac{tr_2(p)tr_2(q)^{1/2}}{\delta}.$$
Combining them, we obtain
$$\frac{tr_2(p)tr_2(q)}{\delta}\leq \mathcal{S}(p*q)^{1/2}\frac{tr_2(p)tr_2(q)^{1/2}}{\delta},$$
i.e. $\mathcal{S}(p*q)\geq tr_2(q)$. Since $\mathcal{S}(p*q)=\mathcal{S}(\overline{p*q})=\mathcal{S}(\overline{q}*\overline{p})$, we see that
$$\max\{tr_2(p),tr_2(q)\}\leq \mathcal{S}(p*q).$$

If $\mathcal{S}(p*q)=tr_2(q)$, we then have the equalities in the inequalities above. Thus we have $p*q=\lambda \mathcal{R}(p*q)$ and $\|p*q\|_2=\frac{tr_2(p)tr_2(q)^{1/2}}{\delta}$, i.e.
$$p*q=\frac{tr_2(p)}{\delta}\mathcal{R}(p*q).$$

If $\frac{\delta}{tr_2(p)}p*q$ is a projection, then
$$\mathcal{S}(p*q)=tr_2(p*q)\frac{\delta}{tr_2(p)}=tr_2(q).$$
\end{proof}

\begin{corollary}
Suppose $\mathscr{P}$ is an irreducible subfactor planar algebra.
Let $p,q$ be projections in $\mathscr{P}_{2,\pm}$ such that $tr_2(p)+tr_2(q)>\delta^2$.
Then $\mathcal{R}(p*q)=\mathcal{R}(p*\overline{q})=1$.
\end{corollary}
\begin{proof}
We assume that $\mathcal{R}(p*q)\neq 1$. Then there is a projection $p_1$  such that $(p*q)p_1=0$.
By Lemma \ref{permu}, we have that $tr_2((\overline{p_1}*p)\overline{q})=0$.
This implies that $(\overline{p_1}*p)\overline{q}=0$, i.e. $\mathcal{R}(\overline{p_1}*p)\overline{q}=0$.
Hence $\delta^2 <tr_2(p)+tr_2(q)\leq \mathcal{S}(\overline{p_1}*p)+tr_2(q)\leq \delta^2$. This leads a contradiction.
Therefore $\mathcal{R}(p*q)=1$.
The equation $\mathcal{R}(p*\overline{q})=1$ can be obtained similarly.
\end{proof}

\begin{corollary}\label{lowerbv}
Suppose that $\mathscr{P}$ is an irreducible subfactor planar algebra.
Let $v, w$ be partial isometries in $\mathscr{P}_{2,\pm}$ such that $\|v*w\|_1=\|v\|_1\|w\|_1$.
Then
$$\max\{tr_2(|v|),tr_2(|w|)\}\leq \mathcal{S}(v*w).$$
Moreover $\mathcal{S}(v*w)=tr_2(|w|)$ if and only if $\frac{\delta}{tr_2(|v|)}v*w$ is a partial isometry.
\end{corollary}
\begin{proof}
The proof is similar to the one of Theorem \ref{lowerb}.
\end{proof}

\begin{remark}
Without the assumption $\|v*w\|_1=\|v\|_1\|w\|_1$, one can find a counterexample in the $\mathbb{Z}_2$ case.
\end{remark}

\begin{corollary}
Suppose $\mathscr{P}$ is an irreducible subfactor planar algebra.
Let $v$ be a partial isometry in $\mathscr{P}_{2,\pm}$ such that $\mathcal{S}(v*\overline{v^*})=tr_2(|v|)$ and $\|v*\overline{v^*}\|_1=\frac{1}{\delta}\|v\|_1^2$.
Then $v$ is a bi-shift of a biprojection.
\end{corollary}
\begin{proof}
By Corollary \ref{lowerbv}, we have that $\frac{\delta}{tr_2(|v|)}v*\overline{v^*}(=x)$ is partial isometry and $tr_2(|x|)=tr_2(|v|)$.
Note that $\mathfrak{F}(x)=\frac{\delta}{tr_2(|v|)}|\mathfrak{F}(v)|^2>0$, thus $\mathfrak{F}(x)$ is extremal.
By Proposition \ref{spec}, $x$ is an extremal bi-partial isometry.
By Main Theorem 2 in \cite{JLW},
$$\mathcal{S}(\mathfrak{F}(x))\mathcal{S}(x)=\delta^2.$$
Note that
$$\mathcal{S}(\mathfrak{F}(x))=\mathcal{S}(\mathfrak{F}(v))\quad \mathcal{S}(x)=\mathcal{S}(v),$$ so
$\mathcal{S}(\mathfrak{F}(v))\mathcal{S}(v)=\delta^2$.
By Main Theorem 2 in \cite{JLW}, $v$ is a bi-shift of a biprojection.
\end{proof}

\begin{corollary}\label{Cor:leftshift}
Suppose $\mathscr{P}$ is an irreducible subfactor planar algebra.
Let $x\in \mathscr{P}_{2,\pm}$ be positive.
If $x*\overline{x}=\frac{tr_2(B)}{\delta}B$ for some biprojection $B$ in $\mathscr{P}$, then $x$ is a multiple of a left shift of $B$.
\end{corollary}
\begin{proof}
Note that $\mathcal{S}(\mathfrak{F}(x))=\mathcal{S}(\mathfrak{F}(B))$.
By Theorem \ref{lowerb}, we have that
$$\mathcal{S}(x)\leq \mathcal{S}(\overline{x}*x)=tr_2(B).$$
Hence $$\mathcal{S}(\mathfrak{F}(x))\mathcal{S}(x)\leq \mathcal{S}(\mathfrak{F}(B))tr_2(B)=\delta^2.$$
By Main Theorems 1 and 2 and Theorem 6.13 in \cite{JLW}, we have that $x$ is a left shift of $B$.
\end{proof}

\begin{proposition}\label{upperb}
Suppose $\mathscr{P}$ is a subfactor planar algebra.
Let $p,q$ be projections  in $\mathscr{P}_{2,\pm}$.
Then
$$\mathcal{S}(p*q))\leq tr_2(p)tr_2(q).$$
\end{proposition}
\begin{proof}
Let $v$ be
$$
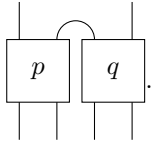
\raisebox{-.7cm}{
\begin{tikzpicture}
\Gc{p}{q}
\end{tikzpicture}}
.$$
Then $vv^*=p*q$ and
$$v^*v=
\raisebox{-1.6cm}
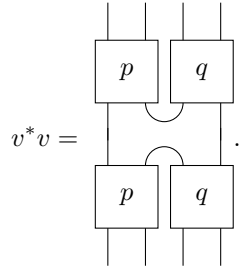
{
\begin{tikzpicture}
\Gc{p}{q}
\begin{scope}[shift={(0,7/2)},yscale=-1,xscale=1]
\Gc{p}{q}
\end{scope}
\end{tikzpicture}}
.$$
Note that $\mathcal{R}(v^*v)\leq p\otimes q$.
Hence $\mathcal{S}(p*q)\leq tr_2(p)tr_2(q).$

\end{proof}

\begin{remark}
Let $p=\sum_{k}p_k$ and $q=\sum_{j}q_j$, where $p_k,q_j$ are projections. Then
\begin{eqnarray*}
\mathcal{S}(p*q)&=&\mathcal{S}((\sum_kp_k)*(\sum_j q_j))\leq \mathcal{S}(\sum_{k,j}p_k*q_j)\\
&\leq& \sum_{k,j} \mathcal{S}(p_k*q_j)\leq \sum_{k,j}tr_2(p_k)tr_2(q_j)=tr_2(p)tr_2(q).
\end{eqnarray*}
Furthermore, if $\mathcal{S}(p*q)=tr_2(p)tr_2(q)$, then $\mathcal{S}(p_k*q_j)=tr_2(p_k)tr_2(q_j).$
\end{remark}

\begin{definition}
Suppose $\mathscr{P}$ is an irreducible subfactor planar algebra.
Let $B$ be a biprojection in $\mathscr{P}_{2,\pm}$.
A projection $q$ in $\mathscr{P}_{2,\pm}$ is said to be a right (left) subshift of the biprojection $B$ if there exists a right (left) shift $B_g$ of $B$ such that $q\leq B_g$.
\end{definition}

\begin{proposition}\label{subshift}
Suppose $\mathscr{P}$ is an irreducible subfactor planar algebra.
Let $B$ be a biprojection in $\mathscr{P}_{2,\pm}$ and $q$ a projection in $\mathscr{P}_{2,\pm}$.
Then
$$\mathcal{R}(q*\overline{q})\leq B \text{ if and only if }q \text{ is a right subshift of }B,$$
and
$$\mathcal{R}(\overline{q}*q)\leq B \text{ if and only if }q \text{ is a left subshift of }B.$$
\end{proposition}
\begin{proof}
Suppose that $\mathcal{R}(q*\overline{q})\leq B$.
Let $p_1=\mathcal{R}(B*q)$. Then $q\leq p_1$.
We shall show that $p_1$ is a right shift of $B$.
Note that
$$p_1*\overline{p_1}\sim B*q*\overline{B*q}=B*q*\overline{q}*B\sim B.$$
We then have $\mathcal{S}(p_1*\overline{p_1})=tr_2(B)$.
By Theorem \ref{lowerb}, we have
$$tr_2(B)\leq tr_2(p_1)\leq \mathcal{S}(p_1*\overline{p_1})=tr_2(B).$$
Therefore $tr_2(p_1)=tr_2(B)$.
By Theorem \ref{lowerb} again, we obtain $\frac{\delta}{tr_2(q)}B*q$ is a projection and $p_1=\frac{\delta}{tr_2(q)}B*q$.
Now $B*p_1=\frac{tr_2(B)}{\delta}p_1$ and $p_1$ is a right shift of $B$.

Suppose $q$ is a right subshift of $B$.
Let $p_1$ be the right subshift of $B$ such that $q\leq p_1$.
Then by Theorem 6.11 in \cite{JLW}, we have
$$q*\overline{q}\leq p_1*\overline{p_1}=\frac{tr_2(B)}{\delta}B$$
i.e. $\mathcal{R}(q*\overline{q})\leq B$.
\end{proof}

\begin{corollary}\label{sub2}
Suppose $\mathscr{P}$ is an irreducible subfactor planar algebra.
Let $p$ be a projection and $B$ a biprojection in $\mathscr{P}_{2,\pm}$.
Then $\mathcal{R}(p*B)$ is a left shift of $B$ if and only if $\mathcal{R}(\overline{p}*p)\leq B$ and $\mathcal{R}(B*p)$ is a right shift of $B$ if and only if $\mathcal{R}(p*\overline{p})\leq B$ .
\end{corollary}

Recall that for projections $p,q$ in $\mathscr{P}_{2,\pm}$, $B_1(q)$ is the biprojection generated by $q*\overline{q}$ and $B_2(p)$ is the spectral projection of $\overline{p}*p$ corresponding to $\frac{tr_2(p)}{\delta}$.

\begin{theorem}\label{lowerm}
Suppose $\mathscr{P}$ is an irreducible subfactor planar algebra.
Let $p,q$ be projections in $\mathscr{P}_{2,\pm}$.
Then the following are equivalent:
\begin{itemize}
\item[(1)] $\mathcal{S}(p*q))=tr_2(p)$;
\item[(2)] $\mathcal{S}(p*(q*\overline{q})^{*(m)}*q^{*(j)})=tr_2(p)$, for some $m\geq 0$, $j\in\{0,1\}$, $m+j>0$, where $q^{*(0)}=e_1$;
\item[(3)] $\mathcal{S}(p*(q*\overline{q})^{*(m)}*q^{*(j)})=tr_2(p)$, for any $m\geq 0$, $j\in\{0,1\}$, $m+j>0$, where $q^{*(0)}=e_1$;
\item[(4)] there exists a biprojection $B$ in $\mathscr{P}_{2,\pm}$ such that $q$ is a right subshift of $B$ and $p=\mathcal{R}(x*B)$ for some $x>0$;
\item[(5)] $B_1(p)\leq B_2(q)$.
\end{itemize}
\end{theorem}
\begin{proof}
$(1)\Rightarrow (4)$: By Theorem \ref{lowerb}, we have that $\frac{\delta}{tr_2(q)}p*q=p_1$ is a projection.
Since
\begin{eqnarray*}
tr_2((p_1*\overline{q})p)&=&tr_2((p*q)p_1)\\
&=&\frac{tr_2(p)tr_2(q)}{\delta}=tr_2(p_1*\overline{q}),
\end{eqnarray*}
by Proposition \ref{holder}, we have that $\mathcal{R}(p_1*\overline{q})\leq p$.
Applying Theorem \ref{lowerb}, we have that
$$tr_2(p)\geq \mathcal{S}(p_1*\overline{q})\geq tr_2(p_1)=tr_2(p)$$
and hence
$$\frac{\delta}{tr_2(q)}p_1*\overline{q}=p, \quad \text{ i.e. } p*(\frac{\delta^2}{tr_2(q)^2}q*\overline{q})=p.$$

Take $B=B_1(q)$. Then by Proposition \ref{Prop:RAS}, we have $p*\frac{\delta}{tr_2(B)}B=p$. Hence $q$ is a right subshift of $B$ and $p=\mathcal{R}(p*B)$.

$(4)\Rightarrow (3)$: Let $p=\mathcal{R}(x*B)$.
Then $p*\frac{\delta}{tr_2(B)}B=p$.
Note that
$$\mathcal{R}((q*\overline{q})^{*(m+j)})\leq \mathcal{R}(B^{*(m+j)})=B.$$

Then by Proposition \ref{idem2}, $p*(\frac{\delta^{m+j}}{tr_2(q)^{m+j}}(q*\overline{q})^{m+j})=p$.
Hence
$$\mathcal{R}(\mathcal{R}(p*(q*\overline{q})^{*(m)}*q^{*(j)})*\overline{q}^{*(j)})=p.$$
By Theorem \ref{lowerb} again, we obtain
\begin{eqnarray*}
tr_2(p)&\leq & \mathcal{S}(p*(q*\overline{q})^{*(m)}*q^{*(j)})\\
&\leq& \mathcal{S}(\mathcal{R}(p*(q*\overline{q})^{*(m)}*q^{*(j)})*\overline{q}^{*(j)})=tr_2(p),
\end{eqnarray*}
i.e.
$$\mathcal{S}(p*(q*\overline{q})^{*(m)}*q^{*(j)}))=tr_2(p).$$

$(3)\Rightarrow (2)$: It is obvious.

$(2)\Rightarrow (1)$: By Theorem \ref{lowerb}, we have that
$$tr_2(p)=\mathcal{S}(p*(q*\overline{q})^{*(m)}*q^{*(j)})\geq \mathcal{S}(p*q)\geq tr_2(p).$$
Hence $\mathcal{S}(p*q)=tr_2(p)$.

$(4)\Leftrightarrow (5)$: By Proposition \ref{subshift}, we have that $B_1(q)\leq B$ if and only if $q$ is a right subshift of $B$.
By Corollary \ref{sub2}, we have that $B\leq B_2(p)$ if and only if $p$ is $\mathcal{R}(x*B)$ for some $x>0$.
\end{proof}

We give sufficient conditions to attain the upper bound of the sum set estimate attained in Proposition \ref{upperb}.
It will be interesting to have a necessary and sufficient description.

\begin{proposition}\label{upperm}
Suppose $\mathscr{P}$ is a subfactor planar algebra.
Let $p,q$ be projections in $\mathscr{P}_{2,\pm}$.
For the following statements:
\begin{itemize}
\item[(1)] $(\overline{p}*p)(q*\overline{q})=\frac{tr_2(p)tr_2(q)}{\delta^2}e_1$.
\item[(2)] $\delta p*q$ is a projection.
\item[(3)] $\mathcal{S}(p*q))=tr_2(p)tr_2(q)$.
\end{itemize}
We have $(1)\Leftrightarrow (2)\Rightarrow (3)$.
\end{proposition}
\begin{proof}
$(1)\Rightarrow (2)$.
We shall show that
$$
\raisebox{-1.6cm}{
\begin{tikzpicture}
\Gc{p}{q}
\begin{scope}[shift={(0,7/2)},yscale=-1,xscale=1]
\Gc{p}{q}
\end{scope}
\end{tikzpicture}}
-\frac{1}{\delta}
\raisebox{-.8cm}{
\begin{tikzpicture}
\foreach \x in {0,1/2}
{
\draw (\x,0)--+(0,1/2);
\draw (\x,1/2+5/6)--+(0,1/2);
}
\draw (-1/6,1/2) rectangle +(5/6,5/6);
\node at (1/4,1/2+5/12) {$p$};
\begin{scope}[shift={(1,0)},yscale=1,xscale=1]
\foreach \x in {0,1/2}
{
\draw (\x,0)--+(0,1/2);
\draw (\x,1/2+5/6)--+(0,1/2);
}
\draw (-1/6,1/2) rectangle +(5/6,5/6);
\node at (1/4,1/2+5/12) {$q$};
\end{scope}
\end{tikzpicture}}
=0
.$$
To see this we will take the square of the left hand side of the equation above first. Expanding the square, we have
\begin{equation}\label{sum0}
\raisebox{-2.2cm}{
\begin{tikzpicture}
\Gc{p}{q}
\begin{scope}[shift={(0,9/2+1/3)},yscale=-1,xscale=1]
\Gc{p}{q}
\end{scope}
\begin{scope}[shift={(0,3/2)}]
\draw (0,0)--+(0,1/2);
\draw (0,1/2+5/6)--+(0,1/2);
\draw (1/2,1/2) arc (-180:0:1/4);
\draw (1/2,1/2+5/6) arc (180:0:1/4);
\draw (-1/6,1/2) rectangle +(5/6,5/6);
\node at (1/4,1/2+5/12) {$p$};
\end{scope}
\begin{scope}[shift={(3/2,3/2)},xscale=-1]
\draw (0,0)--+(0,1/2);
\draw (0,1/2+5/6)--+(0,1/2);
\draw (-1/6,1/2) rectangle +(5/6,5/6);
\node at (1/4,1/2+5/12) {$q$};
\end{scope}
\end{tikzpicture}}
-\frac{2}{\delta}
\raisebox{-1.6cm}{
\begin{tikzpicture}
\Gc{p}{q}
\begin{scope}[shift={(0,7/2)},yscale=-1,xscale=1]
\Gc{p}{q}
\end{scope}
\end{tikzpicture}}
+\frac{1}{\delta}
\raisebox{-.8cm}{
\begin{tikzpicture}
\foreach \x in {0,1/2}
{
\draw (\x,0)--+(0,1/2);
\draw (\x,1/2+5/6)--+(0,1/2);
}
\draw (-1/6,1/2) rectangle +(5/6,5/6);
\node at (1/4,1/2+5/12) {$p$};
\begin{scope}[shift={(1,0)},yscale=1,xscale=1]
\foreach \x in {0,1/2}
{
\draw (\x,0)--+(0,1/2);
\draw (\x,1/2+5/6)--+(0,1/2);
}
\draw (-1/6,1/2) rectangle +(5/6,5/6);
\node at (1/4,1/2+5/12) {$q$};
\end{scope}
\end{tikzpicture}}\;.
\end{equation}
Now the trace of (\ref{sum0}) is
$$tr_2((p*q)(p*q))-\frac{2}{\delta}tr_2(p*q)+\frac{1}{\delta^2}tr_2(p)tr_2(q).$$
By Lemma \ref{permu}, we have that
\begin{eqnarray*}
tr_2((p*q)(p*q))&=&tr_2((q*\overline{q}*\overline{p})\overline{p})=tr_2((\overline{p}*p)(q*\overline{q}))\\
&=&\frac{tr_2(p)tr_2(q)}{\delta^2}=\frac{tr_2(p*q)}{\delta}.
\end{eqnarray*}
Hence the trace of $(\ref{sum0})$ is $0$ , i.e.
$$
\raisebox{-1.6cm}{
\begin{tikzpicture}
\Gc{p}{q}
\begin{scope}[shift={(0,7/2)},yscale=-1,xscale=1]
\Gc{p}{q}
\end{scope}
\end{tikzpicture}}
=\frac{1}{\delta}
\raisebox{-.8cm}{
\begin{tikzpicture}
\foreach \x in {0,1/2}
{
\draw (\x,0)--+(0,1/2);
\draw (\x,1/2+5/6)--+(0,1/2);
}
\draw (-1/6,1/2) rectangle +(5/6,5/6);
\node at (1/4,1/2+5/12) {$p$};
\begin{scope}[shift={(1,0)},yscale=1,xscale=1]
\foreach \x in {0,1/2}
{
\draw (\x,0)--+(0,1/2);
\draw (\x,1/2+5/6)--+(0,1/2);
}
\draw (-1/6,1/2) rectangle +(5/6,5/6);
\node at (1/4,1/2+5/12) {$q$};
\end{scope}
\end{tikzpicture}}\;.
$$
By adding a cap to the middle of the top and the bottom respectively, we have
that $(p*q)^2=\frac{1}{\delta}p*q$, i.e. $\delta p*q$ is a projection.

$(2)\Rightarrow (1)$. Since $\delta p*q$ is a projection, we have that
$$tr_2((\overline{p}*p)(q*\overline{q}))=tr_2((p*q)(p*q))=\frac{tr_2(p)tr_2(q)}{\delta^2}.$$
Reformulating it, we obtain that
$$tr_2((\overline{p}*p-\frac{tr_2(p)}{\delta}e_1)(q*\overline{q}-\frac{tr_2(q)}{\delta}e_1))=0.$$
Note that $\overline{p}*p\geq \frac{tr_2(p)}{\delta}\geq 0$ and $q*\overline{q}\geq \frac{tr_2(q)}{\delta}\geq 0$.
Therefore
$$(\overline{p}*p-\frac{tr_2(p)}{\delta}e_1)(q*\overline{q}-\frac{tr_2(q)}{\delta}e_1)=0,$$
i.e. $(\overline{p}*p)(q*\overline{q})=\frac{tr_2(p)tr_2(q)}{\delta^2}e_1$.

$(2)\Rightarrow (3)$. Since $\delta p*q$ is a projection, $\mathcal{R}(p*q)=\delta p*q$ and $\mathcal{S}(p*q)=\delta tr_2(p*q)=tr_2(p)tr_2(q)$.
\end{proof}

\begin{remark}
In Proposition \ref{upperm}, $(3)$ usually does not implies $(2)$.
For instance, for the $3$-permutation group  $S_3=\{e,(12),(13),(23),(123),(132)\}$. We consider the group subfactor planar algebra $\mathscr{P}^{S_3}$.
Suppose that $\mathscr{P}_{2,+}^{S_2}$ is the group algebra.
Denote by $L_g$ the left multiplication operator.
Let $e_1=\frac{1}{6}\sum_{g\in S_3}L_g,$
$$ p_1=\frac{1}{2}(1+L_{(23)})-e_1,\quad p_2=\frac{1}{2}(1+L_{(13)})-e_1, \quad p_3=\frac{1}{2}(1+L_{(12)})-e_1,$$
and $q=\frac{1}{3}(1+L_{(123)}+L_{(132)})-e_1$.
Then
$tr_2(e_1)=1, tr_2(q_j)=2,j=1,2,3, tr_2(q)=1,$ and
$$\overline{p_j}=p_j,\quad  p_j*p_j=\frac{2}{\sqrt{6}}e_1+\frac{1}{\sqrt{6}}p_j, \quad p_i*p_j=\frac{1}{\sqrt{6}}(\frac{3}{2}-p_i-p_j-e_1)$$
We have that $\mathcal{S}(p_i*p_j)=4=tr_2(p_i)tr_2(p_j)$, but
$$(p_i*p_i)(p_j*p_j)=\frac{tr_2(p_i)tr_2(p_j)}{6}e_1+\frac{1}{6}p_ip_j.$$
\end{remark}

\section{Extremal Pairs of Young's inequality} \label{Sec:Young}

Young initiated the study of Young's inequality on $\mathbb{R}^n$ in 1912 \cite{Young}.
Beckner  proved the sharp Young's inequality for convolution on $\mathbb{R}^n$ and showed that the extremal pairs are Gaussian functions \cite{Beck}.
On unimodular locally compact groups, 
Fournier characterized the extremal pairs of Young's inequality in terms of translations of subcharacters subject to a constraint
\cite{Four}.

In this section we characterize extremal pairs for Young's inequality for subfactors in terms of bi-shifts of biprojections.
We begin with the case that the pair of operators are projections.

\begin{proposition}\label{1pyoung}
Suppose $\mathscr{P}$ is an irreducible subfactor planar algebra.
Let $p,q$ be projections in $\mathscr{P}_{2,\pm}$.
Then the following are equivalent:
\begin{itemize}
\item[(1)] $\|p*q\|_t=\frac{1}{\delta}\|p\|_t\|q\|_1$ for some $1<t<\infty$;
\item[(2)] $\|p*q\|_t=\frac{1}{\delta}\|p\|_t\|q\|_1$ for any $1\leq t\leq \infty$;
\item[(3)]$\mathcal{S}(p*q)=tr_2(p)$.
\end{itemize}
\end{proposition}
\begin{proof}
$(1)\Rightarrow (3)$:
Suppose that $\|p*q\|_t=\frac{1}{\delta}\|p\|_t\|q\|_1$ for some $1<t<\infty$.
Note that $\|p*q\|_\infty\leq \frac{1}{\delta}tr_2(q)$.
By the spectral decomposition, we have
$$\frac{\delta}{tr_2(q)}p*q=p_1+\sum_{j=2}^m\lambda_j p_j,$$
where $\{p_j\}$ is an orthogonal family of projections and $0\leq \lambda_j< 1$ are distinct and $m\geq 1$ is a fixed integer.
By assumption, we have that
$$tr_2(p_1)+\sum_{j=2}\lambda_j^ttr_2(p_j)=tr_2(p).$$
Note that $\|p*q\|_1=\frac{1}{\delta}\|p\|_1\|q\|_1$, i.e.
$$tr_2(p_1)+\sum_{j=2}\lambda_jtr_2(p_j)=tr_2(p).$$
We obtain that $tr_2(p_1)=tr_2(p)$ and $\lambda_j=0$ for $j\geq 2$, i.e. $\mathcal{S}(p*q)=tr_2(p_1)=tr_2(p)$.

$(3)\Rightarrow (2)$:
Suppose that $\mathcal{S}(p*q)=tr_2(p)$.
By Theorem \ref{lowerb}, we have that $\frac{\delta}{tr_2(q)}p*q$ is a projection.
Hence for any $1\leq t\leq \infty$
\begin{eqnarray*}
\frac{\delta}{tr_2(q)}\|p*q\|_t&=&\|\frac{\delta}{tr_2(q)}p*q\|_t=\|\mathcal{R}(p*q))\|_t\\
&=&\mathcal{S}(p*q)^{1/t}=tr_2(p)^{1/t}=\|p\|_t,
\end{eqnarray*}
i.e. $\|p*q\|_t=\frac{1}{\delta}\|p\|_t\|q\|_1$.

$(2)\Rightarrow (1)$: It is obvious.
\end{proof}

\begin{remark}
For the case that $r=\infty$ in Young's inequality, we have the following results.
For any $x$ in $\mathscr{P}_{2,\pm}$, we have that
$$\frac{1}{\delta}tr_2(xx^*)\leq \|x*\overline{x^*}\|_\infty\leq \frac{1}{\delta}\|x\|_2\|\overline{x^*}\|_2=\frac{1}{\delta}\|x\|_2^2,$$
since $\|x*\overline{x^*}\|_\infty\geq \|(x*\overline{x^*})e_1\|_\infty=\frac{1}{\delta}tr_2(xx^*)$.
Hence $\|x*\overline{x^*}\|_\infty=\frac{1}{\delta}\|x\|_2^2$ is true for any $x$ in $\mathscr{P}_{2,\pm}$.

In general, if $\|p*q\|_\infty=\frac{1}{\delta}\|p\|_{t}\|q\|_{s}$ for projections $p,q\in\mathscr{P}_{2,\pm}$, $\frac{1}{t}+\frac{1}{s}=1$, $t,s>1$ , then we have $p*\overline{p}=\overline{q}*q$.
\end{remark}

\begin{proposition}\label{pyoung}
Suppose $\mathscr{P}$ is an irreducible subfactor planar algebra.
Let $p,q$ be projections in $\mathscr{P}_{2,\pm}$.
Then the following are equivalent:
\begin{itemize}
\item[(1)] $\|p*q\|_r=\frac{1}{\delta}\|p\|_{t}\|q\|_{s}$ for some $1<r,t,s<\infty$ such that $\frac{1}{t}+\frac{1}{s}=1+\frac{1}{r}$;
\item[(2)] $\|p*q\|_r=\frac{1}{\delta}\|p\|_{t}\|q\|_{s}$ for any $1\leq r,t,s\leq \infty$ such that $\frac{1}{t}+\frac{1}{s}=1+\frac{1}{r}$;
\item[(3)] there exists a biprojection $B$ in $\mathscr{P}_{2,\pm}$ such that $p$ is a left shift of $B$ and $q$ is a right shift of $B$.
\end{itemize}
\end{proposition}
\begin{proof}
$(1)\Rightarrow (3)$:
By Proposition \ref{young} (Young's inequality), we have that
$$\|p*q\|_r\leq \frac{1}{\delta}\|p\|_1\|q\|_r,\quad \|p*q\|_r\leq \frac{1}{\delta}\|p\|_r\|q\|_1.$$
Hence $tr_2(q)^{\frac{1}{s}-\frac{1}{r}}\leq tr_2(p)^{1-\frac{1}{t}}$ and $tr_2(p)^{\frac{1}{t}-\frac{1}{r}}\leq tr_2(q)^{1-\frac{1}{s}}$, i.e. $tr_2(p)=tr_2(q)$.
Now $\|p*q\|_r=\frac{1}{\delta}tr_2(p)^{1+\frac{1}{r}}$.
By Proposition \ref{1pyoung}, we have that $\mathcal{S}(p*q)=tr_2(p)$.
By Theorem \ref{lowerm}, there exists a biprojection $B$ such that $p=\mathcal{R}(x*B)$ for some $x>0$ and $q$ is a right subshift of $B$.
Since $tr_2(p)=tr_2(q)$, we obtain that $p$ is a left shift of $B$ and $q$ is a right shift of $B$.

$(3)\Rightarrow (2)$:
By Corollary \ref{sub2} and Theorem \ref{lowerb}, we have that
$q*\overline{q}=\frac{tr_2(B)}{\delta}B$.
Hence
$$tr_2(p)\leq \mathcal{S}(p*q)\leq \mathcal{S}(\mathcal{R}(p*q)*\overline{q})=\mathcal{S}(p*q*\overline{q})=tr_2(p).$$
By Theorem \ref{lowerb} again, we see that $\frac{\delta}{tr_2(B)}p*q$ is a projection and $\|p*q\|_r=\frac{1}{\delta}\|p\|_t\|q\|_s$ for any $1\leq r,t,s\leq \infty$.

$(2)\Rightarrow (1)$: It is obvious.
\end{proof}

Next we consider partial isometries.
\begin{proposition}\label{1pv}
Suppose $\mathscr{P}$ is an irreducible subfactor planar algebra.
Let $v,w$ are partial isometries in $\mathscr{P}_{2,\pm}$.
Then the following are equivalent:
\begin{itemize}
\item[(1)] $\|v*w\|_t=\frac{1}{\delta}\|v\|_t\|w\|_1$ for some $1<t<\infty$
\item[(2)] $\|v*w\|_t=\frac{1}{\delta}\|v\|_t\|w\|_1$ for any $1\leq t\leq \infty$
\item[(3)] $\frac{\delta}{tr_2(|w|)}|v*w|$ is a projection and $\|v*w\|_1=\frac{1}{\delta}\|v\|_1\|w\|_1$.
\end{itemize}
\end{proposition}
\begin{proof}
The proof is similar to that of Proposition \ref{1pyoung}.
\end{proof}

\begin{lemma}\label{absolute}
Suppose $\mathscr{P}$ is a subfactor planar algebra.
Let $x, y$ be in $\mathscr{P}_{2,\pm}$.
Then for any $1\leq r\leq \infty$,
$$\|x*y\|_r\leq \||x|*|y|\|_r^{1/2}\||x^*|*|y^*|\|_r^{1/2}.$$
Moreover, if $\|x*y\|_r=\frac{1}{\delta}\|x\|_t\|y\|_s$ for some $1\leq t,s\leq \infty$ and $\frac{1}{r}+1=\frac{1}{t}+\frac{1}{s}$, then
$$\||x|*|y|\|_r=\||x^*|*|y^*|\|_r=\frac{1}{\delta}\|x\|_t\|y\|_s.$$
\end{lemma}
\begin{proof}
Let $\tilde{x}=
\raisebox{-.8cm}{
\begin{tikzpicture}
\begin{scope}[shift={(0,7/2)},yscale=1,xscale=1.5]
\Gc{|x|^{\frac{1}{2}}}{|y|^{\frac{1}{2}}}
\end{scope}
\end{tikzpicture}}\;,
$
 and $\tilde{y}=
\raisebox{-.8cm}{
\begin{tikzpicture}
\begin{scope}[shift={(0,7/2)},yscale=-1,xscale=1.5]
\Gc{w_x|x|^{\frac{1}{2}}}{w_y|y|^{\frac{1}{2}}}
\end{scope}
\end{tikzpicture}}\;
$.
By Proposition \ref{holder} (H\"{o}lder's inequality), we have that
$$\|x*y\|_r=\|\tilde{y}\tilde{x}\|_r\leq \|\tilde{x}\|_{2r}\|\tilde{y}\|_{2r}=\||x|*|y||\|_r^{1/2}\||x^*|*|y^*|\|_r^{1/2}.$$

If $\|x*y\|_r=\frac{1}{\delta}\|x\|_t\|y\|_s$, then
$$\|x*y\|_r\leq \||x|*|y|\|_r^{1/2}\||x^*|*|y^*|\|_r^{1/2}\leq \frac{1}{\delta}\|x\|_t\|y\|_s$$
implies that $\||x|*|y|\|_r=\||x^*|*|y^*|\|_r=\frac{1}{\delta}\|x\|_t\|y\|_s$.
\end{proof}

\begin{proposition}\label{partyoung}
Suppose $\mathscr{P}$ is an irreducible subfactor planar algebra.
Let $v,w$ be partial isometries.
Then the following are equivalent:
\begin{itemize}
\item[(1)] $\|v*w\|_r=\frac{1}{\delta}\|v\|_t\|w\|_s$ for some $1<r,t,s<\infty$ such that $\frac{1}{r}+1=\frac{1}{t}+\frac{1}{s}$;
\item[(2)] $\|v*w\|_r=\frac{1}{\delta}\|v\|_t\|w\|_s$ for any $1\leq r,t,s\leq \infty$ such that $\frac{1}{r}+1=\frac{1}{t}+\frac{1}{s}$;
\item[(3)] both $v$ and $w$ are bi-shifts of biprojections, and $\mathcal{R}((\mathfrak{F}^{-1}(v))^*)=\mathcal{R}(\mathfrak{F}^{-1}(w))$;
\item[(4)] there exists a biprojection $B$ such that $v=(y_v\ _hB)*\mathfrak{F}(\widetilde{B}_g)$ and $w=\mathfrak{F}(\widetilde{B}_g)*(y_w B_f)$, where $B_g,B_f$ are right shifts of $B$, $_hB$ is left shift of $B$ and $y_v, y_w$ are elements in $\mathscr{P}_{2,\pm}$ such that $v,w$ are nonzero partial isometries.
\end{itemize}
\end{proposition}
\begin{proof}
$(1)\Rightarrow (2)$:
Suppose that $\|v*w\|_r=\frac{1}{\delta}\|v\|_t\|w\|_s$ for some $1<r,t,s<\infty$ such that $\frac{1}{r}+1=\frac{1}{t}+\frac{1}{s}$.
By Proposition \ref{young} (Young's inequality), we have
$$\|v*w\|_r\leq \frac{1}{\delta}\|v\|_r\|w\|_1,\quad \|v*w\|_r\leq \frac{1}{\delta}\|v\|_1\|w\|_r,$$
and hence $tr_2(|v|)=tr_2(|w|)$.
By Proposition \ref{1pv}, we see that $\|v*w\|_{\tilde{r}}=\frac{1}{\delta}tr_2(|v|)^{1+\frac{1}{\tilde{r}}}$ for any $1\leq \tilde{r}\leq \infty$.
Therefore $\|v*w\|_{\tilde{r}}=\frac{1}{\delta}\|v\|_{\tilde{t}}\|w\|_{\tilde{s}}$ for any $1\leq \tilde{r},\tilde{t},\tilde{s}\leq \infty$.

$(2)\Rightarrow (3)$:
Let $r=2$. Then $1\leq t,s\leq 2$.
By Proposition \ref{holder} (H\"{o}lder's inequality) and Proposition \ref{hausyoung} (the Hausdorff-Young inequality), we obtain that
\begin{eqnarray*}
\frac{1}{\delta}\|v\|_t\|w\|_s&=&\|v*w\|_2=\|\mathfrak{F}^{-1}(v)\mathfrak{F}^{-1}(w)\|_2\\
&\leq& \|\mathfrak{F}^{-1}(v)\|_{\frac{t}{t-1}}\|\mathfrak{F}^{-1}(w)\|_{\frac{s}{s-1}}\leq \frac{1}{\delta}\|v\|_t\|w\|_s.
\end{eqnarray*}
Hence for any $1\leq t,s\leq 2$
\begin{equation}\label{eqhaus}
\|\mathfrak{F}^{-1}(v)\|_{\frac{t}{t-1}}=\left(\frac{1}{\delta}\right)^{1-\frac{2(t-1)}{t}}\|v\|_t,\quad \|\mathfrak{F}^{-1}(w)\|_{\frac{s}{s-1}}=\left(\frac{1}{\delta}\right)^{1-\frac{2(s-1)}{s}}\|w\|_s,
\end{equation}
\begin{equation}\label{eqholder}
\|\mathfrak{F}^{-1}(v)\mathfrak{F}^{-1}(w)\|_2=\|\mathfrak{F}^{-1}(v)\|_{\frac{t}{t-1}}\|\mathfrak{F}^{-1}(w)\|_{\frac{s}{s-1}}.
\end{equation}
Differentiating Equations (\ref{eqhaus}) with respect to $t,s$ at $t=2,s=2$ respectively, we have that $v$ and $w$ are minimizers of Hirschman-Beckner uncertainty principle for subfactor planar algebras, (see Theorem 5.5 in \cite{JLW}.)
By Main Theorem 2 in \cite{JLW}, we see that $v,w$ are bi-shifts of biprojections.

For Equation (\ref{eqholder}), by Proposition \ref{holder}, we have that
$$|\mathfrak{F}^{-1}(v)|=|\mathfrak{F}^{-1}(w)^*|.$$
Hence $\mathcal{R}((\mathfrak{F}^{-1}(v))^*)=\mathcal{R}(\mathfrak{F}^{-1}(w)).$

$(3)\Rightarrow(4)$:
By the definition of bi-shifts of biprojections, $w$ is a bi-shift of a biprojection means that there is a biprojection $B$ such that $w=\mathfrak{F}(\widetilde{B}_g)*(y_wB_f)$ for some $y_w$, where $\widetilde{B}=\mathcal{R}(\mathfrak{F}(B))$,
$\mathcal{R}(w^*)$ is a right shift $B_f$ of $B$, and $\mathcal{R}(\mathfrak{F}^{-1}(w))$ is a right shift $\widetilde{B}_g$ of $\widetilde{B}$.

Since $v$ is a bi-shift of a biprojection and $\mathcal{R}((\mathfrak{F}^{-1}(v))^*)=\mathcal{R}(\mathfrak{F}^{-1}(w))=\widetilde{B}_g$, using the third form of bi-shifts of biprojections in the Appendix in \cite{JLW}, we have that $v=(y_v\ _hB)*\mathfrak{F}(\widetilde{B}_g)$, where $_hB$ is a left shift of $B$.

$(4)\Rightarrow (2)$:
Since $v,w$ are bi-shifts of the biprojection $B$, we have Equations (\ref{eqhaus}) are true.
By Lemma 6.7 in \cite{JLW}, we have that
$$\mathcal{R}(\mathfrak{F}^{-1}(w))=\mathcal{R}((\mathfrak{F}^{-1}(v))^*)=\widetilde{B}_g.$$
Note that $tr_2(|v|)=tr_2(\ _hB)=tr_2(B_f)=tr_2(|w|)$ and $\mathfrak{F}^{-1}(v),\mathfrak{F}^{-1}(w)$ are multiples of partial isometries.
We see that $|\mathfrak{F}^{-1}(v)|=|\mathfrak{F}^{-1}(w)^*|$ and then
$$\|v*w\|_2=\frac{1}{\delta}\|v\|_t\|w\|_s$$
for any $1\leq t,s\leq 2$ from the argument for "$(2)\Rightarrow (3)$".
By Proposition \ref{1pv}, we have that $\|v*w\|_r=\frac{1}{\delta}\|v\|_r\|w\|_1$ for any $1\leq r\leq \infty$.
Therefore $(2)$ is true.

"$(2)\Rightarrow (1)$". It is obvious.
\end{proof}

Now we will consider the general case.

\begin{proposition}\label{p1p}
Suppose $\mathscr{P}$ is an irreducible subfactor planar algebra.
Let $x,y\in\mathscr{P}_{2,\pm}$.
If $\|x*y\|_t=\frac{1}{\delta}\|x\|_t\|y\|_1$ for some $1<t<2$, then for any $0\leq \Re z\leq 1$
$$\|w_x|x|^{t\frac{1+z}{2}}*y\|_{\frac{2}{1+\Re z}}=\frac{1}{\delta}\|w_x|x|^{t\frac{1+z}{2}}\|_{\frac{2}{1+\Re z}}\|y\|_1.$$
If $\|x*y\|_t=\frac{1}{\delta}\|x\|_t\|y\|_1$ for some $2<t<\infty$, then for any $0\leq \Re z\leq 1$
$$\|w_x|x|^{t\frac{1-z}{2}}*y\|_{\frac{2}{1-\Re z}}=\frac{1}{\delta}\|w_x|x|^{t\frac{1-z}{2}}\|_{\frac{2}{1-\Re z}}\|y\|_1.$$
\end{proposition}
\begin{proof}
Suppose that $\|x\|_t=1$ and $\|y\|_1=\delta$.
When $1<t<2$, we define a complex function $F_1(z)$ given by
$$F_1(z)=tr_2((w_x|x|^{t\frac{1+z}{2}}*y)|x*y|^{t\frac{1-z}{2}}w_{x*y}^*).$$
\begin{eqnarray*}
|F_1(z)|&\leq &\|w_x|x|^{t\frac{1+z}{2}}*y\|_{\frac{2}{1+\Re z}}\||x*y|^{t\frac{1-z}{2}}w_{x*y}^*\|_{\frac{2}{1-\Re z}}\\
&\leq &\frac{1}{\delta}\||x|^{t\frac{1+z}{2}}\|_{\frac{2}{1+\Re z}}\|y\|_1 tr_2(|x*y|^t)^{\frac{1-\Re z}{2}}=1.
\end{eqnarray*}
Hence $F_1(z)$ is a bounded analytic function on $0<\Re z<1$.
Note that
$$F_1(\frac{2}{t}-1)=tr_2((x*y)|x*y|^{t-1}w_{x*y}^*)=1.$$
Therefore $F_1(z)\equiv 1$ on $0\leq \Re z\leq 1$ by the maximum modulus theorem.

When $2<t<\infty$, we consider the function $F_2(z)$ given by
$$F_2(z)=tr_2((w_x|x|^{t\frac{1-z}{2}}*y)|x*y|^{t\frac{1+z}{2}}w_{x*y}^*).$$
Similarly, we can have the proposition proved.
\end{proof}

\begin{proposition}\label{rts}
Suppose $\mathscr{P}$ is an irreducible subfactor planar algebra.
Let $x,y$ be in $\mathscr{P}_{2,\pm}$.
If $\|x*y\|_r=\frac{1}{\delta}\|x\|_t\|y\|_s$ for some $1<r,t,s<\infty$ such that $\frac{1}{r}+1=\frac{1}{t}+\frac{1}{s}$, then for any $-r+1 \leq \Re z\leq r-1$
$$\|w_x|x|^{t\frac{r+1-z}{2r}}*w_y|y|^{s\frac{r+1+z}{2r}}\|_r=\frac{1}{\delta}\|w_x|x|^{t\frac{r+1-z}{2r}}\|_{\frac{2r}{r+1-\Re z}}\|w_y|y|^{s\frac{r+1+z}{2r}}\|_{\frac{2r}{r+1+\Re z}}.$$
\end{proposition}
\begin{proof}
Suppose that $\|x\|_t=\|y\|_s=1$.
We define a function $F(z)$ on $-r+1\leq \Re z\leq r-1$ given by
$$F(z)=tr_2((w_x|x|^{t\frac{r+1+z}{2r}}*w_y|y|^{s\frac{r+1-z}{2r}})|x*y|^{r-1}w_{x*y}^*).$$
\begin{eqnarray*}
|F(z)|&\leq &\|w_x|x|^{t\frac{r+1+z}{2r}}*w_y|y|^{s\frac{r+1-z}{2r}}\|_r\||x*y|^{r-1}\|_{\frac{r}{r-1}}\\
&\leq&\frac{1}{\delta}\|w_x|x|^{t\frac{r+1+z}{2r}}\|_{\frac{2r}{r+1+\Re z}}\|w_y|y|^{s\frac{r+1-z}{2r}}\|_{\frac{2r}{r+1-\Re z}} tr_2(|x*y|^r)^{\frac{r-1}{r}}=\delta^{-r}.
\end{eqnarray*}
Hence $F(z)$ is a bounded analytic function on $-r+1\leq \Re z\leq r-1$.
Since
$$F(\frac{2r}{t}-r-1)=tr_2((x*y)|x*y|^{r-1}w_{x*y}^*)=tr_2(|x*y|^r)=\delta^{-r},$$
we have that $F(z)\equiv 1$ on $-r+1\leq \Re z\leq r-1$ by the maximum modulus theorem.
Therefore we have the proposition proved.
\end{proof}

\begin{proposition}\label{212}
Suppose $\mathscr{P}$ is an irreducible subfactor planar algebra and $x,y\in\mathscr{P}_{2,\pm}$ are positive.
Let $B_1=B_1(x)$ be the biprojection generated by $\overline{x}*x$ and $B_2=B_2(y)$ the spectral projection of $y*\overline{y}$ corresponding to $\frac{\|y\|_2^2}{\delta}$.
Then $\|x*y\|_2=\frac{1}{\delta}\|x\|_1\|y\|_2$ if and only if $B_1\leq B_2$.
Moreover, if $y=\overline{x}$, then $B_1=B_2$ and $x$ is a left shift of $B_1$.
\end{proposition}
\begin{proof}
Recall that $B_2$ is a biprojection by Proposition \ref{spec}.
Note that
\begin{eqnarray*}
\|x*y\|_2^2&=&tr_2((x^**y^*)(x*y))\\
&=&tr_2((\overline{x^*}*x)(y^**\overline{y}))\\
&\leq &\|y^**\overline{y}\|_\infty\|\overline{x^*}*x\|_1\\
&=&\frac{\|y\|_2^2}{\delta}\frac{\|x\|_1^2}{\delta}.
\end{eqnarray*}
If $\|x*y\|_2=\frac{1}{\delta}\|x\|_1\|y\|_2$, we have
$tr_2((\overline{x}*x)(y*\overline{y}))=\|y*\overline{y}\|_\infty\|\overline{x}*x\|_1.$

By Proposition \ref{holder} (H\"{o}lder's inequality), we have that $\mathcal{R}(\overline{x}*x)\leq B_2$.
Hence $B_1\leq B_2$.

If $B_1\leq B_2$, we have $\mathcal{R}(\overline{x}*x)\leq B_2$.
Therefore, we obtain that $\|x*y\|_2=\frac{1}{\delta}\|x\|_1\|y\|_2$ by the argument above.
\end{proof}

\begin{proposition}\label{pp}
Suppose $\mathscr{P}$ is an irreducible subfactor planar algebra.
Let $x,y\in\mathscr{P}_{2,\pm}$ be nonzero positive elements.
If $\|x*y\|_r=\frac{1}{\delta}\|x\|_t\|y\|_s$ for some $1<r,t,s<\infty$ and $\frac{1}{r}+1=\frac{1}{t}+\frac{1}{s}$, then there exists a biprojection $B$ such that $x$ is a multiple of a left shift of $B$ and $y$ is a multiple of a right shift of $B$.
\end{proposition}
\begin{proof}
By Proposition \ref{rts}, we have that
$$\|x^{\frac{t}{r}}*y^s\|_r=\frac{1}{\delta}\|x^{\frac{t}{r}}\|_r\|y^s\|_1,\quad
\|x^{t}*y^{\frac{s}{r}}\|_r=\frac{1}{\delta}\|x^{t}\|_1\|y^{\frac{s}{r}}\|_r.
$$
By Proposition \ref{p1p}, we have that
$$\|x^{\frac{t}{2}}*y^s\|_2=\frac{1}{\delta}\|x^{\frac{t}{2}}\|_2\|y^s\|_1,
\quad
\|x^{t}*y^{\frac{s}{2}}\|_2=\frac{1}{\delta}\|y^{\frac{s}{2}}\|_2\|x^t\|_1.
$$
By Proposition \ref{212}, we have that there exist biprojections $B_1$ and $B_2$ in $\mathscr{P}_{2,\pm}$ such that
\begin{equation}\label{eqran1}
\mathcal{R}(y^s*\overline{y}^s)\leq B_1,\quad (\overline{x}^{\frac{t}{2}}*x^{\frac{t}{2}})B_1=\|\overline{x}^{\frac{t}{2}}*x^{\frac{t}{2}}\|_\infty B_1
\end{equation}
and
\begin{equation}\label{eqran2}
\mathcal{R}(\overline{x}^t*x^t)\leq B_2,\quad (y^{\frac{s}{2}}*\overline{y}^{\frac{s}{2}})B_2=\|y^{\frac{s}{2}}*\overline{y}^{\frac{s}{2}}\|_\infty B_2.
\end{equation}
Therefore, by Lemma \ref{preceq},
$$B_1\leq \mathcal{R}(\overline{x}*x)\leq B_2,\quad B_2\leq \mathcal{R}(y*\overline{y})\leq B_1,$$
i.e. $B_1=B_2=B=\mathcal{R}(\overline{x}*x)=\mathcal{R}(y*\overline{y}).$
From Equation (\ref{eqran1}), we have
$$\overline{x}^{\frac{t}{2}}*x^{\frac{t}{2}}=\|\overline{x}^{\frac{t}{2}}*x^{\frac{t}{2}}\|_\infty B.$$
By Corollary \ref{Cor:leftshift}, we have that $x$ is a left shift of $B$.
Similarly, we obtain that $y$ is a right shift of $B$ from Equation (\ref{eqran2}).
\end{proof}

We characterize the extremal pairs of Young's inequality for subfactor planar algebras.
\begin{theorem}\label{youngeq}
Suppose $\mathscr{P}$ is an irreducible subfactor planar algebra.
Let $x,y$ be in $\mathscr{P}_{2,\pm}$.
Then the following are equivalent:
\begin{itemize}
\item[(1)] $\|x*y\|_r=\frac{1}{\delta}\|x\|_t\|y\|_s$ for some $1<r,t,s<\infty$ such that $\frac{1}{r}+1=\frac{1}{t}+\frac{1}{s}$;
\item[(2)] $\|x*y\|_r=\frac{1}{\delta}\|x\|_t\|y\|_s$ for any $1\leq r,t,s\leq \infty$ such that $\frac{1}{r}+1=\frac{1}{t}+\frac{1}{s}$;
\item[(3)] both $x$ and $y$ are bi-shifts of biprojections, and $\mathcal{R}((\mathfrak{F}^{-1}(x))^*)=\mathcal{R}(\mathfrak{F}^{-1}(y))$;
\item[(4)]there exists a biprojection $B$ such that $x=(a_x\ _hB)*\mathfrak{F}(\widetilde{B}_g)$ and $y=\mathfrak{F}(\widetilde{B}_g)*(a_y B_f)$, where $B_g,B_f$ are right shifts of $B$, $_hB$ is left shift of $B$ and $a_x, a_y$ are elements in $\mathscr{P}_{2,\pm}$ such that $x,y$ are nonzero.
\end{itemize}
\end{theorem}
\begin{proof}
$(1)\Rightarrow (4)$:
By Lemma \ref{absolute}, we have that $\||x|*|y|\|_r=\frac{1}{\delta}\|x\|_t\|y\|_s$.
By Propostion \ref{pp}, we have that $|x|, |y|$ are multiples of projections.
By Proposition \ref{partyoung}, we see $(3)$ is true.

$(4)\Leftrightarrow (3) \Rightarrow (2)$: It is true from Proposition \ref{partyoung}.

$(2)\Rightarrow (1)$: It is obvious.
\end{proof}

For the infinite dimensional case, Kusterman and Vaes introduced locally compact quantum groups \cite{KuVaes}.
Young's inequality for locally compact quantum groups was proved in \cite{LWW}. It would be interesting to characterize the extremal pairs.
People have considered other generalizations of Young's inequalities.
Bobkov and Madiman and Wang conjectured a fractional generalizations of Young inequalities in \cite{BobMadWan}.


\section{Extremal Operators of the Hausdorff-Young Inequality}\label{Sec:Hausdorff-Young}
On unimodular locally compact groups, Russo showed that extremal operators of the Hausdorff-Young inequality are translations of subcharacters \cite{Russo}.

In this section we characterize extremal operators of the Hausdorff-Young inequality for subfactor planar algebras in terms of bi-shifts of biprojections.

\begin{proposition}\label{1tyoungeq}
Suppose $\mathscr{P}$ is an irreducible subfactor planar algebra.
Let $x$ be in $\mathscr{P}_{2,\pm}$.
Then the following are equivalent:
\begin{itemize}
\item[(1)] $\|x*\overline{x^*}\|_r=\frac{1}{\delta}\|x\|_1\|x\|_r$ for some $1<r<\infty$;
\item[(2)] $\|x*\overline{x^*}\|_r=\frac{1}{\delta}\|x\|_1\|x\|_r$ for any $1\leq r\leq \infty$ ;
\item[(3)] $x$ is a bi-shift of a biprojection.
\end{itemize}
\end{proposition}
\begin{proof}
$(1)\Rightarrow (3)$: By Lemma \ref{absolute}, we have
$$\||x|*\left|\overline{x}^*\right|\|_r=\||x|*\overline{|x|}\|_r=\frac{1}{\delta}\|x\|_1\|x\|_r.$$
By Proposition \ref{p1p}, we obtain
$$\||x|^{r/2}*\overline{|x|}|\|_2=\frac{1}{\delta}\|x\|_1\||x|^{r/2}\|_2.$$
By Proposition \ref{212}, we see that
$$B_1(|\overline{x}|)\leq B_2(|x|^{r/2}).$$
Since $B_2(|x|^{r/2})\leq B_1(\overline{|x|}^{r/2})=B_1(\overline{|x|})(=B)$,
we have that $\overline{|x|}^{r/2}*|x|^{r/2}$ is a multiple of $B$.
By Corollary \ref{Cor:leftshift}, $|x|^{r/2}$ is a multiple of a right shift of $B$.
Therefore $x$ is a multiple of a partial isometry.
By Proposition \ref{partyoung}, $x$ is a bi-shift of a biprojection.
\end{proof}

\begin{proposition}\label{hausde}
Suppose $\mathscr{P}$ is an irreducible subfactor planar algebra.
If $\|\mathfrak{F}(x)\|_{\frac{t}{t-1}}=\delta^{1-\frac{2}{t}}\|x\|_t$ for some $1<t<2$, then for any complex number $z$, $0\leq \Re z\leq 1$, we have
$$\|\mathfrak{F}(w_x|x|^{t(1+z)/2})\|_{\frac{2}{1-\Re z}}=\delta^{-\Re z}\|w_x|x|^{t(1+z)/2}\|_{\frac{2}{1+\Re z}}.$$
\end{proposition}
\begin{proof}
We assume that $\|x\|_t=\delta^{1/t}$, $t'=\frac{t}{t-1}$, and consider the function $F(z)$ given by
$$F(z)=tr_2(\mathfrak{F}(w_x|x|^{t(1+z)/2})|\mathfrak{F}(x)|^{t'(1+z)/2}w_{\mathfrak{F}(x)}^*).$$
Since
\begin{eqnarray*}
|F(z)|&\leq &\|\mathfrak{F}(w_x|x|^{t(1+z)/2})\|_{\frac{2}{1-\Re z}}\|\||\mathfrak{F}(x)|^{t'(1+z)/2}w_{\mathfrak{F}(x)}^*\|_{\frac{2}{1+\Re z}}\\
&\leq& \delta^{-\Re z}\|w_x|x|^{t(1+z)/2}\|_{\frac{2}{1+\Re z}}\||\mathfrak{F}(x)|^{t'(1+z)/2}\|_{\frac{2}{1+\Re z}}\\
&=&\delta^{-\Re z}\delta^{\frac{1+\Re z}{2}}\delta^{\frac{1+\Re z}{2}}=\delta,
\end{eqnarray*}
we see that $F(z)$ is a bounded analytic function on $0\leq \Re z\leq 1$.
Note that
$$F(\frac{2}{t}-1)=tr_2(\mathfrak{F}(x)|\mathfrak{F}(x)|^{\frac{1}{t-1}}w_{\mathfrak{F}(x)}^*)=\|\mathfrak{F}(x)\|_{t'}^{t'}=\delta.$$
By the maximum modulus theorem, we have that $F(z)\equiv \delta$ on $0\leq \Re z\leq 1$ and the proposition is proved.
\end{proof}

\begin{theorem}\label{hauseq}
Suppose $\mathscr{P}$ is an irreducible subfactor planar algebra.
Let $x$ be nonzero in $\mathscr{P}_{2,\pm}$.
Then the following are equivalent:
\begin{itemize}
\item[(1)] $\|\mathfrak{F}(x)\|_{\frac{t}{t-1}}=\left(\frac{1}{\delta}\right)^{\frac{2}{t}-1}\|x\|_t$ for some $1< t<2$;
\item[(2)] $\|\mathfrak{F}(x)\|_{\frac{t}{t-1}}=\left(\frac{1}{\delta}\right)^{\frac{2}{t}-1}\|x\|_t$ for any $1\leq t\leq 2$;
\item[(3)] $x$ is a bi-shift of a biprojection.
\end{itemize}
\end{theorem}
\begin{proof}
$(1)\Rightarrow (3)$:
By Proposition \ref{hausde}, we have that
$$\|\mathfrak{F}(w_x|x|^{\frac{3t}{4}})\|_4=\delta^{-1/2}\|w_x|x|^{\frac{3t}{4}}\|_{4/3}.$$
Let $y=w_x|x|^{\frac{3t}{4}}$.
Then
\begin{eqnarray*}
\|y^**\overline{y}\|_2=\||\mathfrak{F}(y)|^2\|_2=\|\mathfrak{F}(y)\|_4^{2}=\delta^{-1}\|y\|_{4/3}^2.
\end{eqnarray*}
By Theorem \ref{youngeq}, we have that $y$ is a bi-shift of a biprojection and  so is $x$.

$(3)\Rightarrow (2)$: It can be checked directly.

$(2)\Rightarrow (1)$: It is obvious.
\end{proof}

\section{Block Maps}\label{Sec:Block Maps}
\subsection{Block Maps for subfactors}
In this section, we introduce block maps for subfactors motivated by the renormalization group of 2D lattice models and by the square relation of bi-shifts of biprojections. 
We study their dynamic systems and prove that the limit points are all multiples of biprojections and zero. The asymptotical phenomenon of block maps coincides with the scaling limit of 2D lattice models as explained in the introduction.

We briefly recall the 2D Ising model as a supplement to the explanation in the introduction. 
We do not give the detailed computation here.  Instead, we explain it as a motivation of the definition of the block maps.

A configuration in the 2D Ising model is an assignment of the spins $\pm$ to vertices of a 2D lattice.
For the ferromagnet Ising model, the nearest neighborhood interaction $J$ is defined as 
\begin{align*}
J(++)&=J(--)=-1, \\
J(+-)&=J(-+)=1.
\end{align*}
 The total energy $H$ is the sum of $J$ over all edges.
 The partition function is 
 \begin{align*}
 Z=\sum_\sigma e^{-\beta H},
 \end{align*}
 summing over all configurations $\sigma$, $\beta=T^{-1}$ is the inverse temperature. 

In the graph planar algebra of the bipartite graph $A_3$, \scalebox{.5}{
\tikz{
\draw (0,0)--(2,0);
\fill (0,0) circle (.2);
\fill[white] (1,0) circle (.2);
\draw (1,0) circle (.2);
\fill (2,0) circle (.2);
}},
the partition function can be represented by a planar diagram $D$ shown in Fig~\ref{Lattice to Graph}.
The shaded/unshaded regions are assigned to a black/white vertex in $A_3$, corresponding to the  spins $\pm$.
The crossing in $D$ are labelled by the 2-box $B(T,J)=(e^{\beta}-e^{-\beta})\ID+e^{-\beta} \sqrt{2} \JP$, which represents the interaction. More precisely, when the two black boundaries of $B(T,J)$  are assigned to the same spin, then $B(T,J)$ contributes to a scalar $e^{\beta}$. If the spins are different, then the scalar is $e^{-\beta}$. Given an assignment, the product of those scalars is $e^{-\beta H}$. In total, the diagram $D$ defines the partition function $Z$ multiplying $2^{\frac{n}{2}}$, where $n$ is the number of vertices in the lattice. 

The renormalization group for lattice models is used to rescale the size of the 2D lattice by a factor 2 as follows:
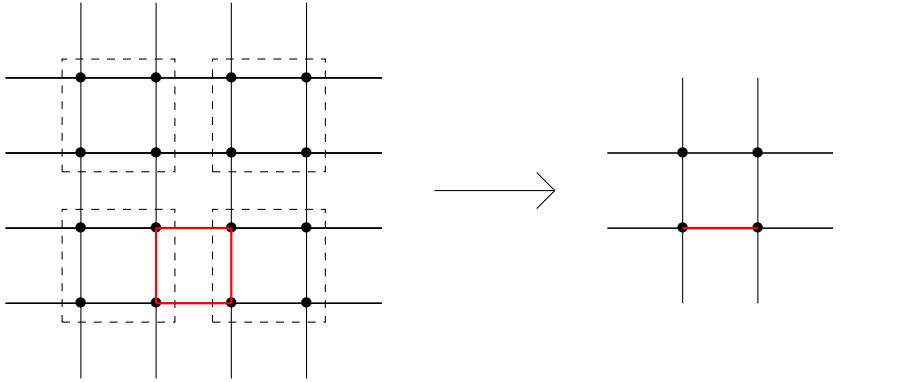
\ben
\begin{tikzpicture}
\foreach \x in {0,1,2,3}
{
\draw (\x,-1)--+(0,5);
\foreach \y in {0,1,2,3}
{
\draw (-1,\y)--+(5,0);
\node at (\x,\y) {$\bullet$};
}}

\foreach \x in {0,2}
{
\foreach \y in {0,2}
{
\draw[dashed] (\x-1/4,\y-1/4) rectangle +(3/2,3/2);
}}

\draw[thick,red] (1,0)--(2,0);
\draw[thick,red] (1,0)--(1,1);
\draw[thick,red] (1,1)--(2,1);
\draw[thick,red] (2,0)--(2,1);

\begin{scope}[shift={(6.3,1.5)},xscale=.8,yscale=.8]
\draw (0,0)--+(-2,0);
\draw (0,0)--+(-.3,.3);
\draw (0,0)--+(-.3,-.3);
\end{scope}
\begin{scope}[shift={(7,0)}]
\foreach \x in {1,2}
{
\draw (\x,0)--+(0,3);
\foreach \y in {1,2}
{
\draw (0,\y)--+(3,0);
\node at (\x,\y) {$\bullet$};
}
\draw[thick,red] (2,1)--(1,1);
}
\end{scope}
\node at (11,1) {.};
\end{tikzpicture}
\een 
An edge in the rescaled lattice is coming from four edges of a square in the original lattice. In planar algebras, the renormalization procedure combines four 2-boxes to one.

The square relation of a bi-shift of a biprojection $w$ (Theorem 6.11 in \cite{JLW}) says that:
\begin{align*}
(\overline{w^*}*w)(\overline{w}*w^*)&=\frac{\|w\|_2^2}{\delta}(\overline{w}\overline{w^*})*(w^*w).
\end{align*}
Diagrammatically,
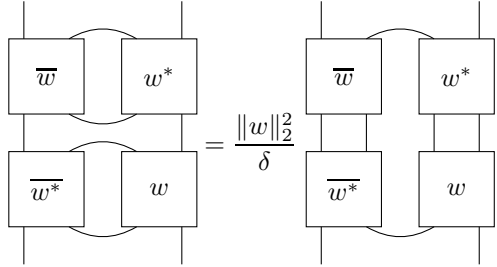
\begin{align*}
\raisebox{-1.5cm}{
\begin{tikzpicture}
\draw (0,0) rectangle (1,1);
\node at (0.5, 0.5) {$\overline{w^*}$};
\begin{scope}[shift={(1.5,0)}]
\draw (0,0) rectangle (1,1);
\node at (0.5, 0.5) {$w$};
\end{scope}
\begin{scope}[shift={(0,1.5)}]
\draw (0,0) rectangle (1,1);
\node at (0.5, 0.5) {$\overline{w}$};
\end{scope}
\begin{scope}[shift={(1.5,1.5)}]
\draw (0,0) rectangle (1,1);
\node at (0.5, 0.5) {$w^*$};
\end{scope}
\draw (.2,0)--++(0,-.5);
\draw (.2,2.5)--++(0,.5);
\draw (2.3,0)--++(0,-.5);
\draw (2.3,2.5)--++(0,.5);
\draw (.2,1)--++(0,.5);
\draw(2.3,1)--++(0,.5);
\draw (.8,0) to [bend left=-30] (1.7,0);
\draw (.8,2.5) to [bend left=30] (1.7,2.5);
\draw (.8,1) to [bend left=30] (1.7,1);
\draw (.8,1.5) to [bend left=-30] (1.7,1.5);
\end{tikzpicture}}
&=\frac{\|w\|_2^2}{\delta}
\raisebox{-1.5cm}{
\begin{tikzpicture}
\draw (0,0) rectangle (1,1);
\node at (0.5, 0.5) {$\overline{w^*}$};
\begin{scope}[shift={(1.5,0)}]
\draw (0,0) rectangle (1,1);
\node at (0.5, 0.5) {$w$};
\end{scope}
\begin{scope}[shift={(0,1.5)}]
\draw (0,0) rectangle (1,1);
\node at (0.5, 0.5) {$\overline{w}$};
\end{scope}
\begin{scope}[shift={(1.5,1.5)}]
\draw (0,0) rectangle (1,1);
\node at (0.5, 0.5) {$w^*$};
\end{scope}
\draw (.2,0)--++(0,-.5);
\draw (.2,2.5)--++(0,.5);
\draw (2.3,0)--++(0,-.5);
\draw (2.3,2.5)--++(0,.5);
\draw (.2,1)--++(0,.5);
\draw(2.3,1)--++(0,.5);
\draw (.8,0) to [bend left=-30] (1.7,0);
\draw (.8,2.5) to [bend left=30] (1.7,2.5);
\draw (.8,1)--++(0,.5);
\draw (1.7,1)--++(0,.5);
\end{tikzpicture}}
\end{align*}
Both sides combine four 2-boxes to one and we consider them as block maps.

\begin{definition}\label{Def:RM}
We define block maps $B_{\lambda}$, $0\leq \lambda \leq 1$, on $x \in \mathscr{P}_{2,\pm}$ by
\begin{align*}
B_{cm}(x)&=\frac{\delta^2}{\|x\|_1\|x\|_2^2}(\overline{x^*}*x)(\overline{x}*x^*) \;,\\
B_{mc}(x)&=\frac{\delta}{\|x\|_\infty\|x\|_2^2}(\overline{x}\overline{x^*})*(x^*x) \;,\\
B_{\lambda}&=\lambda B_{cm}+(1-\lambda) B_{mc} \;.
\end{align*}
\end{definition}

It is mentioning that Jones studied the renormalization procedure on 1D quantum spin chains while investigating the reconstruction program from subfactors to CFT \cite{Jon14}. Our motivation and definition are different.

\begin{definition}
An positive operator is called bi-positive, if its Fourier transform is also positive.
\end{definition}

Note that $B(T,J)$ are bi-positive operators.

\begin{proposition}\label{Prop:bi-positive}
For a non-zero element $x$ in $\mathscr{P}_{2,\pm}$, $B_{\lambda}(x)$ is bi-positive, $0\leq \lambda \leq 1$.
Therefore the space of bi-positive operators is invariant under the action of block maps.
\end{proposition}

\begin{proof}
Note that $\overline{x^*}*x=(\overline{x}*x^*)^*$, so $B_{cm}(x)$ is positive.
By Schur product Theorem \ref{Thm:Schur}, $B_{mc}(x)$ is positive.
Then $\F B_{cm}(x)$ is a positive scalar multiple of $B_{mc} (\overline{\F(x)^*}) $, so it is positive and  $B_{cm}(x)$ is bi-positive.
Similarly  $B_{mc}(x)$ is bi-positive. So $B_{\lambda}(x)$ is bi-positive.
\end{proof}

\begin{proposition}\label{ReFour}
For any bi-positive operator $x\in\mathscr{P}_{2,\pm}$, $0\leq \lambda \leq 1$, we have
\begin{align*}
\F B_{\lambda}(x)=B_{1-\lambda}\F(x).
\end{align*}
In particular, $\F B_{1/2}(x)=B_{1/2}\F(x)$.
\end{proposition}

\begin{proof}
Since $x$ is bi-positive, we have that
\begin{align*}
\|\F(x)\|_{\infty}&=\frac{\|x\|_1}{\delta} \;,\\
\frac{\|\F(x)\|_1}{\delta}&=\|x\|_{\infty} \;.
\end{align*}
Note that the SFT $\F$ on $\mathscr{P}_{2,\pm}$ is a $90^\circ$ rotation, and $\|\F(x)\|_2=\|x\|_2$, so
the conjugation of $\F$ switches $B_{cm}$ and $B_{mc}$. Therefore,
$\F B_{\lambda}(x)=B_{1-\lambda}\F(x)$.
\end{proof}

\begin{proposition}\label{DyIneq}
Suppose $\mathscr{P}$ is an irreducible subfactor planar algebra.
Then for any $1\leq t\leq \infty$, $0\leq \lambda \leq 1$,
\begin{align*}
\|B_{\lambda}(x)\|_t\leq \|x\|_t.
\end{align*}
\end{proposition}

\begin{proof}

By Proposition \ref{holder} (H\"{o}lder's inequality) and Proposition \ref{young} (Young's inequality), we have
\begin{eqnarray}
\|B_{cm}(x)\|_t
&\leq &\frac{\delta^2}{\|x\|_1\|x\|_2^2}\|\overline{x^*}*x\|_\infty\|\overline{x}*x^*\|_t \notag\\
&= &\frac{\delta^2}{\|x\|_1\|x\|_2^2}\frac{\|x\|_2^2}{\delta}\|\overline{x}*x^*\|_t \notag\\
&\leq &\frac{\delta}{\|x\|_1}\frac{\|\overline{x}\|_1\|x^*\|_t}{\delta}=\|x\|_t.\label{last}
\end{eqnarray}
and
\begin{eqnarray*}
\|B_{mc}(x)\|_t&\leq &\frac{\delta}{\|x\|_\infty\|x\|_2}\||x|\|_t\||x|^2\|_1\\
&\leq &\frac{\delta}{\|x\|_\infty}\|x\|_\infty\|x\|_t=\|x\|_t.
\end{eqnarray*}
Therefore
\begin{align}\label{Equ:1}
\|B_{\lambda}(x)\|_t \leq \lambda \|B_{cm}(x)\|_t + (1-\lambda)\|B_{mc}(x)\|_t \leq \|x\|_t.
\end{align}
\end{proof}

\begin{proposition}\label{teq}
Suppose $\mathscr{P}$ is an irreducible subfactor planar algebra and $x$ is nonzero in $\mathscr{P}_{2,\pm}$.
Then $\|B_{cm}(x)\|_t=\|x\|_t$ (or $\|B_{mc}(x)\|_t=\|x\|_t$) for some $1<t<\infty$ if and only if $x$ is an extremal bi-partial isometry.
Moreover, $B_{\lambda}(x)=x$, for some $0\leq \lambda \leq 1$, if and only if $x$ is a positive multiple of a biprojection.
\end{proposition}
\begin{proof}

Then $\|B_{\lambda}(x)\|_t=\|x\|_t$ for some $1<t<\infty$, $0\leq \lambda \leq 1$, if and only if $x$ is an extremal bi-partial isometry.

By Theorem 6.13 in \cite{JLW}, we have that if $x$ is an extremal bi-partial isometry, then $B_{cm}(x)$ is a multiple of a biprojection and $\|B_{cm}(x)\|_t=\|x\|_t$ for any $1\leq t\leq \infty$.

If $\|B_{cm}(x)\|_t=\|x\|_t$ for some $1<t<\infty$, by Inequality \eqref{last} in Lemma \ref{DyIneq}, we have
$$\|\overline{x}*x^*\|_t=\frac{1}{\delta}\|x\|_1\|x\|_t.$$
By Proposition \ref{1tyoungeq}, we have that $x$ is a bi-shift of a biprojection.

Hence we have that $\|B_{cm}(x)\|_t=\|x\|_t$ for some $1<t<\infty$ if and only if $x$ is a bi-shift of a biprojection.
Similarly, we have that $\|B_{mc}(x)\|_t=\|x\|_t$ for some $1<t<\infty$ if and only if $x$ is a bi-shift of a biprojection.

If $x$ is a multiple of a biprojection, we have $B_{\lambda}(x)=x$ by definitions.
Conversely if $B_{\lambda}(x)=x$, then $\|B_{cm}(x)\|_t=\|x\|_t$ or $\|B_{mc}(x)\|_t=\|x\|_t$, for any $1\leq t\leq \infty$, by Inequality \eqref{Equ:1} in Proposition \ref{DyIneq}. So $x$ is an extremal bi-partial isometry.
Moreover, $x$ is bi-positive by Proposition \ref{Prop:bi-positive}. So $x$ is a multiple of a biprojection.
\end{proof}

\begin{theorem}\label{Thm:RM}
Suppose $\mathscr{P}$ is an irreducible subfactor planar algebra.
Then for any $x\in\mathscr{P}_{2,\pm}$ and $B_\lambda$ on $\mathscr{P}_{2,\pm}$, $0\leq \lambda\leq 1$, the sequence $\{B^b_{\lambda}(x)\}_{n\geq 1}$ converges to $0$ or a multiple of a biprojection.
\end{theorem}

\begin{proof}
By Proposition \ref{DyIneq}, take $\displaystyle \ell=\lim_{n\to \infty}\|B_{\lambda}^n(x)\|_2$, the limit of the decreasing sequence.
If $\ell=0$, then $B_{\lambda}^n (x)$ converges to $0$.

If $\ell\neq 0$, we assume that $\ell=1$ by the linearity.
Since $\mathscr{P}_{2,\pm}$ is finite dimensional,
there is an accumulation point $z$ of $\{B_{\lambda}^n (x)\}_{n\in \mathbb{N}}$.
Then $B_{\lambda}(z)$ is also an accumulation point and $\|B_{\lambda}z\|_2=\|z\|_2=\ell$.
By Proposition \ref{Prop:bi-positive}, $B_{\lambda}^n(x)$ is bi-positive, so is $z$. Hence $z$ is a multiple of a biprojection by Proposition \ref{teq}.

Let us prove the uniqueness of the accumulation point.
It is known that there are finitely many intermediate subfactors of an irreducible subfactor \cite{Longo, Xu16}.
That means there are finitely many biprojections in $\mathscr{P}_{2,\pm}$. Take
\begin{align*}
m&=\min\left\{ \left. \left\|\frac{B}{\|B\|_2} - \frac{B'}{\|B'\|_2} \right\|_2 \right| B , B' \text{ are distinct biprojections} \right\}.
\end{align*}

Since $\mathscr{P}_{2,\pm}$ is finite dimensional, so $t$-norm topologies are equivalent for $1\leq t\leq \infty$.
By the continuity of multiplication and coproduct, $B_{\lambda}$ is uniformly continuous in the $2$-norm unit ball of $\mathscr{P}_{2,\pm}$.
Hence there exists $\kappa>0$, such that
\begin{align}
\|B_{\lambda}(a)-B_{\lambda}(b)\|_2 &< \frac{m}{2},
\end{align}
whenever $\|a-b\|_2\leq \kappa$ , $\|a\|_2\leq 1$ and $\|b\|_2\leq 1$.

If there is an accumulation points different from $z$,
then there exist a subsequence $\{n_k\}_{k\in \mathbb{N}}$ such that
\begin{align*}
\|B_{\lambda}^{n_k}(x) -z\|_2   &\leq \epsilon \;,\\
\|B_{\lambda}^{n_k+1}(x)-z \|_2 &>\epsilon \;.
\end{align*}
Then
\begin{align*}
\|B_{\lambda}^{n_k+1}(x)-B_{\lambda}(z) \|_2 &< \frac{m}{2}.
\end{align*}
Recall that the accumulation point $z$ is a multiple of a biprojection, so $B_{\lambda}(z)=z$ and
\begin{align*}
\epsilon<\|B_{\lambda}^{n_k+1}(x)-z \|_2< \frac{m}{2}.
\end{align*}

Therefore there is an accumulation point $z'$ of $\{B_{\lambda}^{n_k+1}(x)\}_{k\in \mathbb{N}}$, such that
\begin{align*}
\epsilon\leq \|z'-z\|_2 \leq \frac{m}{2}.
\end{align*}
By the above argument, $z'$ is also a multiple of a biprojection, and $\|z'\|_2=1$. It is a contradiction. Therefore the accumulation point $z$ of $\{B_{\lambda}^n (x)\}_{n\in \mathbb{N}}$ is unique.
\end{proof}

Those limit points are minimizers the Hirschman-Beckner uncertainty principle, see Theorem 5.5 in  \cite{JLW}. We conjecture that 
\begin{conjecture}
The block maps reduce the Hirschman-Beckner entropy of $x$, i.e., the von-Neumann entropy of $|x|^2 \oplus |\F(x)|^2$ .
\end{conjecture}

By the linearity, the block map $B_{\lambda}$ is well-defined on the projective space $\mathscr{P}_{2,\pm}/\mathbb{C}$.
For the $\mathbb{Z}_2$ case, a bi-positive 2-box is given by $a \ID +b \JP$, for $a,b\geq 0$, and at least one is positive. The projective space is parameterized by $t=\frac{b}{a}$ and $t\in[0,\infty]$.

We show that there are three fixed points of $B_{1/2}$ in the projective space.
Among the three, $t=0$ and $t=\infty$ are stable corresponding to the two biprojections and the two limit cases of the Ising model. The third one $t=1$ is not stable corresponding to critical temperature of the Ising model. 

Kramers and Wannier observed that if the critical temperature is unique, then it has to be the fixed point of the Kramers-Wannier duality. This duality switches vertical edges and horizontal edges in the lattice. In planar algebras, it becomes a $90^{\circ}$ rotation of the 2-box $B(T,J)$, namely the SFT of $B(T,J)$. Thus $B(T,J)$ has to be invariant under the action of SFT at the critical temperature, if it is unique. That means $t=1$. 

Furthermore, we study the dynamic system of $B_{1/2}$ on $\mathscr{P}_{2,\pm}$. We find a gap between the limit and zero when $t\neq 1$. It is gapless when $t=1$. This result is stated in theorem \ref{Thm:mass gap} and we prove it here.

\begin{proof}[Proof of Theorem \ref{Thm:mass gap}]
Let us define $B_{1/2}^n(a\ID+b\JP)=a_n\ID+b_n\JP$, $\displaystyle t=\frac{b}{a}$ and $\displaystyle t_n=\frac{b_n}{a_n}$.

By the definition of $B_{1/2}$,  
\begin{align*}
B_{1/2}(a\ID+b\JP)&=\frac{(a^2\delta+2ab)^2(b\delta+a)+(a^4\delta+2a^2b^2\delta+4a^3b)(a\delta+b)}{2(a\delta+b)(b\delta+a)(a^2\delta+b^2\delta+2ab)}\ID \\
&+ \frac{(b^4\delta+2a^2b^2\delta+4ab^3)(b\delta+a)+(b^2\delta+2ab)^2 (a\delta+b)}{2(a\delta+b)(b\delta+a)(a^2\delta+b^2\delta+2ab)}\JP .
\end{align*}

Then 
\begin{align*}
t_1&=\frac{(t^4\delta+2t^2\delta+4t^3)(t\delta+1)+(t^2\delta+2t)^2 (\delta+t)}{(\delta+2t)^2(t\delta+1)+(\delta+2t^2\delta+4t)(\delta+t)}\\
&= t+t(t^2-1)\frac{2\delta^2(t^2+1)+(\delta^3+3\delta )t}{2\delta^2 +(\delta^3+9\delta)t+(6\delta^2+8)t^2+6\delta t^3} .
\end{align*}

If $t=1$, then $t_n=1$, for $n=1,2,\ldots$. By theorem \ref{Thm:RM}, $\displaystyle \lim_{n\to \infty} B_{1/2}^n(a\ID+b\JP)=0$. 

If $0<t<1$, then $t_n\to 0$. When $t$ is small, we have $a'=a+t (4-\delta-\delta^{-1}) +o(t^2)>a$. Thus $\displaystyle \lim_{n\to \infty} B_{1/2}^n(a\ID+b\JP)= c_1 \ID$, for some $c_1>0$.

If $t>1$, then $t_n\to \infty$. By Proposition \ref{ReFour}, we have 
$\displaystyle \lim_{n\to \infty} B_{1/2}^n(a\ID+b\JP)= c_2 \JP$, for some $c_2>0$.

\end{proof}

\subsection{Infinite dimensional cases}

Note that we can also define the block maps on $\mathbb{R}^n$, locally compact groups, or locally compact quantum groups in general.
We give a short discussion here about the block maps on $\mathbb{R}^n$.
We propose some questions related to $\mathbb{R}^n$ and general cases.

For subfactors, bi-shifts of biprojections are all minimizers of the Hirschman-Beckner uncertainty principles.
The bi-positive ones are (scalar multiples of) biprojections. They are limit points of the block maps.

On $\mathbb{R}^n$, the minimizers of the Hirschman-Beckner uncertainty principles are Gaussian functions. The bi-positive ones are $g=ce^{\frac{-a|x|^2}{2}}$, $a, c > 0$, on $\mathbb{R}^n$.
We observe that they are eigenfunctions of the block maps:
\begin{align*}
B_{\lambda}(g)=\frac{1}{2^n}g.
\end{align*}
The eigenvalue $\frac{1}{2^n}$ tells the dimension of $\mathbb{R}^{n}$.
The rate of convergence of the renormalization procedure is fast by a computer test, but we do not have a mathematical control.
We conjecture that

\begin{conjecture}
For any $f\in L^{\infty}(\mathbb{R}^n) \cap L^{1}(\mathbb{R}^n) \cap L^{2}(\mathbb{R}^n)$,
$f$ converges to a Gaussian function under the action of the block map $2^nB_{\lambda}$.
\end{conjecture}

We expect that the block maps provide a mechanism to approximate ``Gaussian functions'' on locally compact (quantum) groups.
We propose the following questions:

(1) Whether the bi-positive minimizers of the Hirschman-Beckner uncertainty principles are eigenvectors of the block map, if they exist?

(2) What are the possible eigenvalues?

(3) Whether these bi-positive minimizers are the only stable limit points in the projective space under the dynamic action of the block map?

%
%

\section{Concluding Remarks}\label{Sec:Concluding Remarks}

\subsection{A comparisons between commutative cases and non-commutative cases}

The functions on a finite group form a commutative algebra.
Different from the group case, the pair of $C^*$-algebras arising from a subfactor are noncommutative (and non-cocommutative) in general. The topology is discrete for finite groups, but not for the pair of $C^*$-algebras.
Moreover, the minimal projections in the $C^*$-algebra could be non-group-like.

Because of these differences, it is not obvious how to formalize the concepts from the group case to subfactors. Many methods on commutative algebras are not suitable for subfactors, since these methods rely on the commutative condition and the group-like property.

For groups, Young's inequality, uncertainty principles, and sum set estimates were proved independently.
This is not the case for subfactors. Instead the three topics are mixed together as shown in Fig.~\ref{Fig:Fourier}.

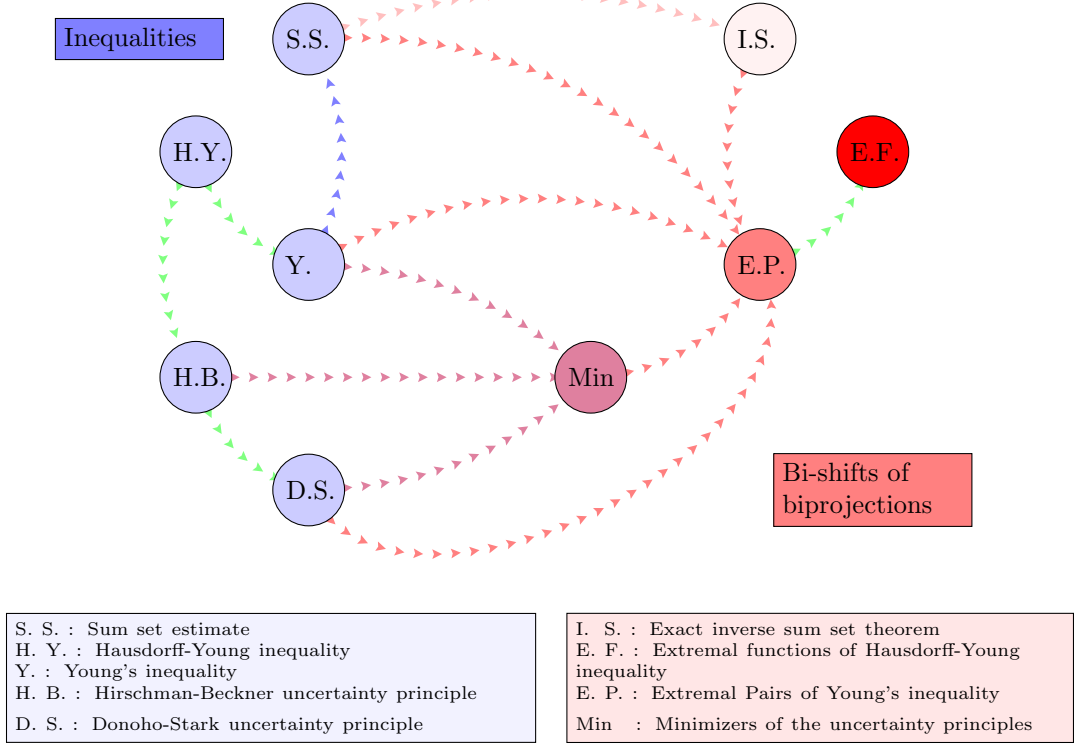
\begin{figure}[h]
\begin{tikzpicture}
\begin{scope}[shift={(1,0)},xscale=.75,yscale=.75]
\node at (0,0) {0};
\coordinate (HY) at (0,0);
\coordinate (Y) at (2,-2);
\coordinate (HB) at (0,-4);
\coordinate (DS) at (2,-6);
\coordinate (Min) at (7,-4);
\coordinate (SS) at (2,2);
\coordinate (IS) at (10,2);
\coordinate (EP) at (10,-2);
\coordinate (EF) at (12,0);
\fill[decoration={shape backgrounds,shape=dart,shape size=1.3mm}, fill=green!50] decorate{ (HY) to[bend right] (HB)};
\fill[decoration={shape backgrounds,shape=dart,shape size=1.3mm}, fill=green!50] decorate{ (HY) to[bend right] (Y)};
\fill[decoration={shape backgrounds,shape=dart,shape size=1.3mm}, fill=green!50] decorate{ (HB) to[bend right] (DS)};
\fill[decoration={shape backgrounds,shape=dart,shape size=1.3mm}, fill=blue!50] decorate{ (Y) to[bend right] (SS)};
\fill[decoration={shape backgrounds,shape=dart,shape size=1.3mm}, fill=red!50] decorate{ (IS) to[bend right] (EP)};
\fill[decoration={shape backgrounds,shape=dart,shape size=1.3mm}, fill=red!50] decorate{ (Min) to[bend right] (EP)};
\fill[decoration={shape backgrounds,shape=dart,shape size=1.3mm}, fill=green!50] decorate{ (EP) to[bend right] (EF)};
\fill[decoration={shape backgrounds,shape=dart,shape size=1.3mm}, fill=red!50] decorate{ (Y) to[bend right=-30] (EP)};
\fill[decoration={shape backgrounds,shape=dart,shape size=1.3mm}, fill=red!50] decorate{ (SS) to[bend right=-30] (EP)};
\fill[decoration={shape backgrounds,shape=dart,shape size=1.3mm}, fill=pink!100] decorate{ (SS) to[bend right=-20] (IS)};
\fill[decoration={shape backgrounds,shape=dart,shape size=1.3mm}, fill=purple!50] decorate{ (Y) to[bend right=-20] (Min)};
\fill[decoration={shape backgrounds,shape=dart,shape size=1.3mm}, fill=purple!50] decorate{ (HB) to[bend right=0] (Min)};
\fill[decoration={shape backgrounds,shape=dart,shape size=1.3mm}, fill=purple!50] decorate{ (DS) to[bend right=20] (Min)};
\fill[decoration={shape backgrounds,shape=dart,shape size=1.3mm}, fill=red!50] decorate{ (DS) to[bend right=90] (EP)};
\node[text width=.6cm, decorate, fill=blue!20,draw,circle] at (HY) {H.Y.};
\node[text width=.6cm, decorate, fill=blue!20,draw,circle] at (Y)  {Y.};
\node[text width=.6cm, decorate, fill=blue!20,draw,circle] at (HB) {H.B.};
\node[text width=.6cm, decorate, fill=blue!20,draw,circle] at (DS) {D.S.};
\node[text width=.6cm, decorate, fill=purple!50,draw,circle] at (Min) {Min};
\node[text width=.6cm, decorate, fill=blue!20,draw,circle] at (SS) {S.S.};
\node[text width=.6cm, decorate, fill=pink!20,draw,circle] at (IS) {I.S.};
\node[text width=.6cm, decorate, fill=red!50,draw,circle] at (EP) {E.P.};
\node[text width=.6cm, decorate, fill=red!100,draw,circle] at (EF) {E.F.};
\node[text width=2cm, decorate, fill=blue!50,draw,rectangle] at (-1,2) {Inequalities};
\node[text width=2.4cm, decorate, fill=red!50,draw,rectangle] at (12,-6) {Bi-shifts of biprojections};
\end{scope}
\node[text width=6.8cm, decorate, fill=blue!5,draw,rectangle] at (2,-7)
{
\scriptsize
S. S. : Sum set estimate \\
H. Y. : Hausdorff-Young inequality \\
Y.~~~ : Young's inequality \\
H. B. : Hirschman-Beckner uncertainty principle \\
D. S. : Donoho-Stark uncertainty principle
};
\node[text width=6.5cm, decorate, fill=red!10,draw,rectangle] at (9.3,-7)
{
\scriptsize
I.\; S. : Exact inverse sum set theorem
\\ E. F. : Extremal functions of Hausdorff-Young inequality
\\ E. P. : Extremal Pairs of Young's inequality
\\ Min~\; : Minimizers of the uncertainty principles
};
\end{tikzpicture}

\caption{It shows the logic chain of our proofs of the results about the Fourier analysis on subfactors.} \label{Fig:Fourier}
\end{figure}

\subsection{Applications}
One can recover groups and their duals from group subfactors. 
One can also obtain results about subgroups and double cosets from group-subgroup subfactors.
Some other examples that are not group-like are from finite dimensional $C^*$-Hopf algebras (or Kac algebras), quantum doubles \cite{Dri86,Ocn91,Pop94,LonReh95,Mug03II}, and Jones-Wassermann subfactors of unitary modular tensor categories \cite{Xu00,LiuXu}.

A simple result on finite abelian group may be non-trivial in other cases. The sum set estimate \eqref{Sumsetestimate1} is obvious on the set of group elements, but it is non-trivial on the representations of a finite group.
If we apply the sum set estimate to group subfactors and take $p,q$ to be central projections of the group algebra acting on the regular representation, then Equation \eqref{Sumsetestimate} is equivalent to the following result about representations of a finite group:

\begin{corollary}\label{Cor:sumset}
  Suppose $G$ is a finite group, and $V,W$ are finite dimensional representations of $G$.
  Then $|V|\leq |V\otimes W| \leq |V| |W|$, where $|V|$ is the sum of the square of the dimension of irreducible sub representations of $V$.
\end{corollary}

The SFT has been applied to quantum information in \cite{JLWholo,LWJ-quon}.
\footnote{We call it the string Fourier transform to distinguish from other Fourier transforms in quantum information.}
It is shown that the SFT on quons for a unitary modular tensor category is the modular $S$ matrix in \cite{LiuXu,LWJ-quon}.
Therefore we can apply our results on Fourier analysis to quons and obtain new results about the modular tensor category and the $S$ matrix.
For example, the sumset estimate \eqref{Sumsetestimate} implies that 
\begin{corollary}\label{Cor:sumset}
  Suppose $\mathscr{C}$ is a unitary modular tensor category, and $V,W$ are two objects of $\C$.
  Then $|V|\leq |V\otimes W| \leq |V| |W|$, where $|V|$ is the sum of the square of the quantum dimension of irreducible subobjects of $V$.
\end{corollary}
We will discuss the Fourier analysis of the $S$ matrix in forthcoming work.

\subsection{Characterizations of bi-shifts of biprojections}

Combining the results in \cite{JLW}, we list numbers of characterization of a bi-shift of a biprojection in an irreducible subfactor planar algebra $\mathscr{P}$.

\begin{theorem}\label{Thm:bishifts}
Let $x$ be nonzero in $\mathscr{P}_{2,\pm}$.
Then the following are equivalent:
\begin{itemize}
\item[(1)] $x$ is a bi-shift of a biprojection;
\item[(2)] $x$ is an extremal bi-partial isometry;
\item[(3)] $\mathcal{S}(x)\mathcal{S}(\mathfrak{F}(x))=\delta^2$;
\item[(4)] $H(|x|^2)+H(|\mathfrak{F}(x)|^2)=\|x\|_2^2(2\log \delta-4\log\|x\|_2)$;
\item[(5)] $\|x*\overline{x^*}\|_r=\frac{1}{\delta}\|x\|_t\|x\|_s$ for some $1<r,t,s<\infty$ such that $\frac{1}{r}+1=\frac{1}{t}+\frac{1}{s}$;
\item[(6)] $\|x*\overline{x^*}\|_r=\frac{1}{\delta}\|x\|_1\|x\|_r$ for some $1<r<\infty$;
\item[(7)] $\|x*\overline{x^*}\|_r=\frac{1}{\delta}\|x\|_t\|x\|_s$ for any $1\leq r,t,s\leq \infty$ such that $\frac{1}{r}+1=\frac{1}{t}+\frac{1}{s}$;
\item[(8)] $\|\mathfrak{F}(x)\|_{\frac{t}{t-1}}=\left(\frac{1}{\delta}\right)^{\frac{2}{t}-1}\|x\|_t$ for some $1< t<2$;
\item[(9)] $\|\mathfrak{F}(x)\|_{\frac{t}{t-1}}=\left(\frac{1}{\delta}\right)^{\frac{2}{t}-1}\|x\|_t$ for any $1\leq t\leq 2$;
\item[(10)] $x$ is a multiple of a partial isometry, $\mathcal{S}(x*\overline{x^*})=\mathcal{S}(x)$, and $\|x*\overline{x^*}\|_1=\|x\|_1^2$.
\end{itemize}
\end{theorem}

\end{document}